\documentclass[10pt,reqno]{amsart}

\usepackage{amssymb,amsrefs,amscd,amsmath}
\usepackage{a4wide,graphicx,nicefrac}
\usepackage{todonotes}
\usepackage[shortlabels]{enumitem}
\usepackage{mathrsfs}
\usepackage{hyperref}
\usepackage{ourbib}
\usepackage[all]{xy}
\usepackage{callouts}

\setlist{leftmargin=8mm}

\hypersetup{
    colorlinks,
    linkcolor={red!50!black},
    citecolor={blue!50!black},
    urlcolor={blue!80!black}
}

\newcommand{\RR}{\mathscr{R}}
\newcommand{\K}{\mathscr{K}}
\newcommand{\U}{\mathscr{U}}
\newcommand{\X}{\mathscr{X}}
\newcommand{\II}{\mathrm{II}} 
\newcommand{\tr}{\mathrm{tr}} 
\newcommand{\R}{\mathbb{R}}
\newcommand{\N}{\mathbb{N}}
\newcommand{\dM}{{\partial M}}
\newcommand{\D}{\mathsf{D}}
\newcommand{\eps}{\varepsilon}
\renewcommand{\phi}{\varphi}
\newcommand{\Ahat}{\mathsf{\hat A}}
\newcommand{\ch}{\mathsf{ch}}
\renewcommand{\c}{\mathsf{c}}
\newcommand{\scal}{\mathrm{scal}}
\newcommand{\ric}{\mathrm{ric}}
\renewcommand{\div}{\mathrm{div}}
\newcommand{\<}{\left\langle}
\renewcommand{\>}{\right\rangle}
\newcommand{\dV}{\,\mathsf{dV}}
\newcommand{\vol}{\mathsf{vol}}
\newcommand{\dA}{\mathsf{dA}}
\newcommand{\nor}{\mathsf{nor}}
\newcommand{\myicon}{$\,\,\,\triangleright$}

\DeclareMathOperator{\ind}{\mathrm{ind}}
\DeclareMathOperator{\id}{\mathrm{id}}
\DeclareMathOperator{\Diff}{\mathrm{Diff}} 

\newtheorem{theorem}{Theorem}[section]
\newtheorem{lemma}[theorem]{Lemma}
\newtheorem{proposition}[theorem]{Proposition}
\newtheorem{corollary}[theorem]{Corollary}

\theoremstyle{definition}
\newtheorem{remark}[theorem]{Remark}
\newtheorem{definition}[theorem]{Definition} 
\newtheorem{example}[theorem]{Example}

\begin{document}

\title{Boundary conditions for scalar curvature} 
\dedicatory{Dedicated to Misha Gromov on the occasion of his 77$^\mathrm{th}$ birthday}
\author{Christian B\"ar}
\address{Universit\"at Potsdam, Institut f\"ur Mathematik, 14476 Potsdam, Germany}
\email{\href{mailto:cbaer@uni-potsdam.de}{cbaer@uni-potsdam.de}}
\urladdr{\url{https://www.math.uni-potsdam.de/baer}}
\author{Bernhard Hanke}
\address{Universit\"at Augsburg, Institut f\"ur Mathematik, 86135 Augsburg, Germany}
\email{\href{mailto:hanke@math.uni-augsburg.de}{hanke@math.uni-augsburg.de}}
\urladdr{\url{https://www.math.uni-augsburg.de/diff/hanke}}

\begin{abstract} 
Based on  the Atiyah-Patodi-Singer index formula, we construct an obstruction to positive scalar curvature metrics with mean convex boundaries on spin manifolds of infinite $K$-area. 
We also characterize the extremal case.

Next we show a general deformation principle for boundary conditions of metrics with lower scalar curvature bounds.
This implies that  the relaxation of boundary conditions often induces weak homotopy equivalences of spaces of such metrics.
This can be used to refine the  smoothing of codimension-one singularites \`a la Miao and  the deformation of boundary conditions \`a la Brendle-Marques-Neves, among others. 

Finally, we construct compact manifolds for which the spaces of positive scalar curvature metrics with mean convex boundaries have nontrivial higher homotopy groups. 

Video abstract: \url{https://vimeo.com/530490182}
\end{abstract}

\keywords{Manifolds with boundary, lower scalar curvature bounds, lower mean curvature bounds, Atiyah-Patodi-Singer index formula, infinite $K$-area, area-enlargeability, deformations of Riemannian metrics, Min-Oo conjecture, spaces of positive scalar curvature metrics with conditions on the second fundamental form of the boundary}

\subjclass[2010]{53C21, 53C23; Secondary: 53C24, 53C27, 58D17, 58J20}

\maketitle
\tableofcontents

\section{Introduction} 
It is well known that  not every closed manifold $M$ carries Riemannian metrics of positive scalar curvature and, if it does, that the space of all such metrics on $M$ carries a rich topology in general. 
On the other hand, if $M$ is compact, connected, of dimension $2$ at least and has \emph{nonempty boundary}, then, by Gromov's $h$-principle, the space of positive scalar curvature metrics on $M$ is nonempty and contractible. 
Hence, in order to encounter interesting phenomena similar to the closed case, one needs to add conditions on the metric along the boundary. 
The paper at hand shows that this works for numerous boundary conditions of geometric importance. 

For some time the focus was laid  almost entirely on positive scalar curvature metrics which are of product type near the boundary. 
This is due to the fact that the Atiyah-Patodi-Singer index theorem for manifolds with boundary was formulated in this setting in \cite{APS1} for convenience.
Moreover, it can be technically useful to replace the manifold with boundary by one without boundary by attaching a cylinder.

Having product structure near the boundary is, however, not the most obvious assumption from a geometric perspective. 
It implies, for instance, that the metric induced on the boundary is also of positive scalar curvature, thus ruling out many topologies.
Consequently, the consideration of less restrictive boundary conditions, such as mean convex or totally geodesic boundaries, has recently been promoted at various occasions, notably in the work of Gromov  \cites{GL1980, GromovDiracPlateau, Gromov2018, Gromov2019, Gromov2020, GromovMean}.
One motivation for this development lies in the endeavour to find a comparison geometric characterization of lower scalar curvature bounds, see Li \cite{Li} for an example regarding nonnegative scalar curvature on   $3$-dimensional polyhedra. 

One of our main  results, Theorem~\ref{master}, describes  a general deformation scheme  for the strengthening of boundary conditions while preserving lower scalar curvature bounds on smooth manifolds with compact boundaries.
As a rule of thumb such deformations exist whenever they are mean curvature nonincreasing with respect to the interior normal along the boundary, much in the spirit of Gromov's  ``Bending Lemma'', see   \cite{Gromov2018}*{p.~705}. 

As a notable application, Theorem~\ref{master} implies that any positive scalar curvature metric with mean convex boundary can be deformed through such metrics to a positive scalar curvature metric with totally geodesic boundary. 
In contrast to previous results such as Carlotto-Li \cite{CarlottoLi}*{Prop.~1.4} our deformations exist in compact families such that metrics whose boundaries are already totally geodesic will keep this property during the deformation. 
The last feature is crucial for the investigation of spaces of metrics in later parts of our work. 

The  deformations in Theorem~\ref{master}  are constructed as follows:  
In a first step, we create an arbitrarily large scalar curvature contribution in a small neighborhood of the boundary by adjusting the second derivative  of the given metric in the normal direction along the boundary while leaving its  $1$-jet along the boundary constant and decreasing the scalar curvature over the whole manifold by an arbitrarily small amount, see Proposition~\ref{step1}. 
Here we use  local flexibility lemma  of B\"ar-Hanke \cite{BH}*{Thm.~1.2 and Addendum~3.4}.
In a second step, we deform the $1$-jet along the boundary of the resulting metric with the help of an explicit cut-off function which is supported in an arbitrarily small neighborhood of the boundary, see Proposition~\ref{schwer}. 
The crucial point is that the decrease of scalar curvature in the second deformation may be bounded independently of the first deformation.
By an appropriate choice of the first deformation one can hence assume that the concatenation of the two deformations produces a scalar curvature decrease which is only due to an application of the local flexibility lemma. 
One may speculate  that a similar strategy can also be used to refine the recent approximation results in Chow \cite{Chow2020}. 

Our paper is organized as follows. 
In Section \ref{nonexistence} we generalize the well-known index theoretic obstruction for positive scalar curvature metrics on closed spin manifolds, which is based on the  Lichnerowicz formula and the Atiyah-Singer index theorem, to compact manifolds with mean convex boundaries.
We use the Atiyah-Patodi-Singer index formula, see Theorem~\ref{thm:nonexist}. 
Our discussion also includes a characterization of the extremal case. 
We remark that the APS-index formula has also been used in recent work of Lott \cite{Lott2020} for generalizing the Llarull and Goette-Semmelmann rigidity theorems, see   \cite{Llarull} and \cite{GS2002}, to manifolds with boundaries. 

In Section~\ref{spaces} we prove our main deformation result, Theorem~\ref{master}, along the lines sketched above.
 
Our results are applied in Section~\ref{applications} to spaces of metrics with lower scalar curvature bounds under various boundary conditions. 
Theorems~\ref{main}, \ref{main_suppl}, \ref{mainwithboundary} and \ref{mainwithboundaryaddendum} describe a number of instances of geometric relevance when inclusions of such spaces are weak homotopy equivalences. 
Furthermore, we study spaces of metrics with lower scalar curvature bounds and mean-convex singularities along hypersurfaces and refine the well-known approximation results of Miao \cite{Miao}, see Theorem~\ref{desingularize}.
Subsection~\ref{minoo} revisits the construction of counterexamples to the Min-Oo conjecture by Brendle-Marques-Neves \cite{BMN} in the light of our deformation results. 

In Subsection~\ref{homgr} we construct, for each $m \geq 0$, examples of manifolds with nonempty boundaries for which the spaces of positive scalar curvature metrics and mean convex boundaries have nontrivial homotopy in degree $m$. 
These are the first examples of this kind.
Using the preceding results of our paper this implies analogous properties for spaces of positive scalar curvature metrics with other boundary conditions. 

\medskip

\textit{Acknowledgments:} 
We are grateful to Alessandro Carlotto and Jan Metzger for useful conversations. 
We thank the Deutsche Forschungsgemeinschaft for the financial support by the SPP 2026 ``Geometry at Infinity'' and the Mathematisches Forschungsinstitut Oberwolfach where this work was initiated under ideal working conditions.

\section{Nonexistence of metrics with positive scalar curvature and mean convex boundary} \label{nonexistence} 

Througout this section $M$ will be a compact Riemannian manifold with boundary $\dM$ and of dimension $n\ge2$. 

\subsection{Harmonic spinors on manifolds with boundary} \label{harm_spinor} 

Let $M$ carry a spin structure.
We denote the complex spinor bundle by $\Sigma M$ and fix a Hermitian vector bundle $E\to M$ with a compatible connection.
Then the twisted Dirac operator $D_E$ is a first-order elliptic differential operator which acts on sections of $\Sigma M\otimes E$.

If the dimension $n$ of $M$ is even, the spinor bundle splits into spinors of positive and negative chirality, $\Sigma M = \Sigma^+M\oplus \Sigma^-M$.
The Dirac operator then interchanges chirality, i.e., with respect to the splitting $\Sigma M\otimes E = (\Sigma^+M\otimes E)\oplus (\Sigma^-M\otimes E)$ it has the block form
$$
D_E = 
\begin{pmatrix}
0 & D_E^- \\
D_E^+ & 0
\end{pmatrix}
.
$$
Both $\Sigma^+M|_\dM$ and $\Sigma^-M|_\dM$ can be naturally identified with $\Sigma\dM$.
We denote the twisted Dirac operator on $\dM$ by $D_E^\dM$.
Since $\dM$ is a closed manifold, the operator $D_E^\dM$ is essentially selfadjoint. 

For any Borel subset $I\subset \R$ denote by $\chi_I\colon \R\to \{0,1\}$ the characteristic function of $I$.
Functional calculus for selfadjoint operators on the Hilbert space $L^2(\dM,\Sigma\dM\otimes E)$ provides us with the projections $\chi_I(D_E^\dM)$.

On the Sobolev space $W^1(M,\Sigma M\otimes E)$ the restriction map to the boundary is well defined and yields bounded linear maps $W^1(M,\Sigma^\pm M\otimes E)\to L^2(\dM,\Sigma\dM\otimes E)$.
We say that $\phi\in W^1(M,\Sigma^\pm M\otimes E)$ satisfies the \emph{APS-boundary conditions} if $\chi_{[0,\infty)}(D_E^\dM)(\phi|_\dM)=0$.
Similarly, we say that $\phi\in W^1(M,\Sigma M\otimes E)$ satisfies the \emph{weak APS-boundary conditions} if $\chi_{(0,\infty)}(D_E^\dM)(\phi|_\dM)=0$.
Obviously, if $\ker(D_E^\dM)=0$ then the APS and the weak APS-boundary conditions coincide.

Consider the operators 
\begin{gather}
D_E^{+,\mathsf{APS}} \colon  \{\phi\in W^1(M,\Sigma^+ M\otimes E): \, \phi \mbox{ satisfies APS-boundary conditions}\} \to L^2(M, \Sigma^-M)
\label{eq:defDAPS} \\
D_E^{-,\mathsf{wAPS}} \colon  \{\phi\in W^1(M,\Sigma^- M\otimes E): \, \phi \mbox{ satisfies weak APS-boundary conditions}\} \to L^2(M, \Sigma^+M)
\label{eq:defDsAPS} 
\end{gather}
Both $D_E^{+,\mathsf{APS}}$ and $D_E^{-,\mathsf{wAPS}}$ are Fredholm operators and the index formula of Atiyah, Patodi, and Singer \cite{APS1}*{Thm.~4.2} says that
\begin{align}
\ind(D_E^{+,\mathsf{APS}}) 
& = 
-\ind(D_E^{-,\mathsf{wAPS}}) = \dim(\ker(D_E^{+,\mathsf{APS}})) - \dim(\ker(D_E^{-,\mathsf{wAPS}})) \notag\\
&= 
\int_M \Ahat(M) \wedge \ch(E) + \int_\dM T(\Ahat(M) \wedge \ch(E)) - \frac{\dim(\ker(D_E^\dM)) + \eta(D_E^\dM)}{2}.
\label{eq:APS}
\end{align}
Here $\Ahat(M)$ is the $\Ahat$-form built out of the curvature of the Levi-Civita connection on $M$, $\ch(E)$ is the Chern character form constructed from the curvature $R^E$ of $E$, $T(\Ahat(M) \wedge \ch(E))$ is the corresponding transgression form and $\eta(D_E^\dM)$ denotes the $\eta$-invariant defined using the spectrum of $D_E^\dM$.
See \cite{APS1} for details.

Next we want to find criteria which ensure that the index of $D_E^{+,\mathsf{APS}}$ vanishes.
This is based on the \emph{Lichnerowicz formula} (see e.g.\ \cite{LM}*{Thm.~8.17}):
\begin{equation}
D_E^2 = \nabla^*\nabla + \frac{\scal}{4}\cdot\id + \K^E
\label{eq:Lichnerowicz}
\end{equation}
where $\scal$ denotes the scalar curvature of $M$ and $\K^E$ is the symmetric curvature endomorphism on $\Sigma M\otimes E$ induced by $R^E$,
\begin{equation}
\K^E (\sigma\otimes e) = \sum_{1\leq i<j\leq n} f_i \cdot f_j \cdot \sigma \otimes R^E(f_i,f_j) e.
\label{eq:KE}
\end{equation}
Here $f_1,\ldots,f_n$ is an orthonormal basis of the tangent space of $M$ at the base point of $\sigma\otimes e$.
The tangent vectors act by Clifford multiplication on $\sigma$.

Let $p\in M$. 
We define the \emph{operator norm} of $R^E$ at $p$ by
$$
|R^E_p| := \max\{|R^E(f_1\wedge f_2)e| : f_1,f_2\in T_pM \mbox{ and }e\in E_p\mbox{ with }|f_i|=|e|=1\}.
$$

\begin{lemma}\label{lem:KE}
Let $p\in M$.
Then all eigenvalues $\lambda$ of $\K^E_p$ satisfy
$$
|\lambda| \leq \tfrac{n(n-1)}{2}|R^E_p| .
$$
\end{lemma}

\begin{proof}
Let $\phi\in\Sigma_pM\otimes E_p$ be an eigenvector of $\K^E_p$ for the eigenvalue $\lambda$.
Fix $1\le i < j\le n$.
The endomorphism $\Sigma_pM \to \Sigma_pM$, $\sigma\mapsto f_i\cdot f_j\cdot\sigma$, is skew-symmetric and has square $-\id$.
Thus there exists an orthonormal basis $\sigma_1,\sigma_1',\sigma_2,\sigma_2',\ldots$ of $\Sigma_pM$ such that $f_i\cdot f_j\cdot\sigma_k=\sigma_k'$ and $f_i\cdot f_j\cdot\sigma_k'=-\sigma_k$.
Write $\phi=\sum_k (\sigma_k\otimes e_k + \sigma_k'\otimes e_k')$ for suitable $e_k,e_k'\in E_p$.
Then $|\phi|^2 = \sum_k (|e_k|^2+|e_k'|^2)$.
We compute
\begin{align}
|\<(f_i\cdot f_j\otimes R^E(f_i,f_j))\phi,\phi\>|
&=
\big|\sum_{k\ell}\<\sigma_k'\otimes R^E(f_i,f_j)e_k - \sigma_k\otimes R^E(f_i,f_j)e_k' , \sigma_\ell\otimes e_\ell + \sigma_\ell'\otimes e_\ell'\>\big| \notag\\
&=
\big|\sum_k\< R^E(f_i,f_j)e_k,e_k'\>-\< R^E(f_i,f_j)e_k',e_k\>\big| \notag\\
&\le
2\sum_k |R^E_p| |e_k| |e_k'| \notag\\
&\le
\sum_k|R^E_p| (|e_k|^2+ |e_k'|^2) \notag\\
&=
|R^E_p| |\phi|^2 . \notag
\end{align}
Summing over $i$ and $j$ yields
$$
|\lambda| |\phi|^2
=
|\<\K^E_p\phi,\phi\>|
\le
\tfrac{n(n-1)}{2}|R^E_p| |\phi|^2
$$
which proves the lemma.
\end{proof}

Put $\|R^E\| := \max_{p\in M} |R^E_p|$.

\begin{corollary}\label{cor:0pos}
Assume
\begin{equation}
\min_M \scal \ge 2n(n-1) \|R^E\|.
\label{eq:scalK}
\end{equation}
Then the $0$-zero part of the right hand side of \eqref{eq:Lichnerowicz}, $\frac14 \scal  \id + \K^E$, is everywhere positive semidefinite.
If the inequality in \eqref{eq:scalK} is strict, then $\frac14\scal\id + \K^E$ is everywhere positive definite.
\hfill$\Box$
\end{corollary}

This now leads to a vanishing result for the kernel of the Dirac operator with weak APS-boundary conditions.
Denote by $\nu$ the interior unit normal field of $\dM$.
Let $H\colon \dM\to\R$ be the mean curvature of $\dM$ with respect to $\nu$.

\begin{proposition}\label{prop:vanish}
Let $M$ be a connected compact Riemannian spin manifold and $E\to M$ a Hermitian vector bundle with compatible connection.
Assume \eqref{eq:scalK} and $H\geq0$.
Furthermore, let there be a point $p\in M$ such that $H(p)>0$ or $\scal(p) > 2n(n-1)|R^E_p|$.
Then
$$
\ind(D_E^{+,\mathsf{APS}}) = \dim\ker(D_E^\mathsf{wAPS}) = 0.
$$
\end{proposition}

\begin{proof}
Let $\phi\in \ker(D_E^\mathsf{wAPS})$.
We need to show that $\phi=0$.
It is known that $\phi$ is smooth up to the boundary, see e.g.\ \cite{BB}*{Cor.~7.18}.
By Corollary~\ref{cor:0pos}, the endomorphism field $\K^E_S := \frac{\scal}{4}\cdot\id + \K^E$ is positive semidefinite.
The Lichnerowicz formula \eqref{eq:Lichnerowicz} and an integration by parts yield
\begin{align*}
0 
&= 
\int_M\<D_E^2\phi,\phi\> \dV \\
&=
\int_M(\<\nabla^*\nabla\phi,\phi\> + \<\K^E_S\phi,\phi\>) \dV \\
&=
\int_M(|\nabla\phi|^2 + \<\K^E_S\phi,\phi\>) \dV + \int_\dM \<\nabla_\nu\phi,\phi\> \dA .
\end{align*}
Here $\dV$ is the volume element on $M$ and $\dA$ that on $\dM$.
In order to control the boundary term, we use the relation
$$
-\nu\cdot D\phi = D_E^\dM\phi + \nabla_\nu\phi - \tfrac{n-1}{2}H\phi
$$
which holds along the boundary, see e.g.\ \cite{B96}*{Prop.~2.2}.
Since $\phi$ is harmonic this implies
$$
\int_\dM \<\nabla_\nu\phi,\phi\> \dA = - \int_\dM \<D_E^\dM\phi,\phi\> \dA + \frac{n-1}{2}\int_\dM H|\phi|^2\dA
$$
and hence
\begin{equation} \label{eq:Lichn_bound} 
0= \int_M|\nabla\phi|^2 \dV + \int_M\<\K^E_S\phi,\phi\> \dV   - \int_\dM \<D_E^\dM\phi,\phi\> \dA + \frac{n-1}{2}\int_\dM H|\phi|^2\dA.
\end{equation} 
All four summands on the right hand side are nonnegative; 
the second one because of Corollary~\ref{cor:0pos}, the third one because we imposed weak APS-boundary conditions, and the last one because of $H\ge0$.
Thus all four terms must be zero.
In particular, $\phi$ is parallel and hence $|\phi|$ is constant.
If $\phi\neq0$, then we conclude that $H\equiv0$ and that $\phi$ lies everywhere in the kernel of $\K^E_S$.
Thus $\K^E_S$ is nowhere positive definite.
\end{proof}

Combining Proposition~\ref{prop:vanish} with \eqref{eq:APS} yields

\begin{corollary}\label{cor:vanish}
Let $M$ be a connected compact Riemannian spin manifold and $E\to M$ a Hermitian vector bundle with compatible connection.
If $H\ge0$ and $2n(n-1) \|R^E\| < \min_M \scal$ then
$$
\int_M \Ahat(M) \wedge \ch(E) + \int_\dM T(\Ahat(M) \wedge \ch(E)) - \frac{\dim(\ker(D_E^\dM)) + \eta(D_E^\dM)}{2} = 0.
$$
\end{corollary}

\subsection{\texorpdfstring{Infinite $K$-area}{Infinite K-area}} \label{infiniteKarea} 

Our nonexistence proof for metrics with certain properties is based on the concept of $K$-area as introduced in \cite{G1996}*{Sec.~4}.
For manifolds without boundary $K$-area has been investigated in \cites{G1996,Hanke2012,Fukumoto2015,Hunger2019}.  

We call a Hermitian vector bundle $E$ over $M$ with connection \emph{admissible} if it is isomorphic to the trivial bundle with trivial connection over a neighborhood of the boundary and it has at least one nontrivial Chern number.
The latter means that there are $\gamma_j\in\N_0$ such that 
$$
\int_M \c_{\gamma_1}(E) \wedge \cdots \wedge \c_{\gamma_m}(E) \neq 0.
$$
Here $\c(E) = \c_0(E)+\c_1(E)+\ldots+\c_m(E) = 1+\c_1(E)+\ldots+\c_m(E)$ is the Chern form of $E$.
Admissible bundles can exist only on even-dimensional manifolds because $\c_j(E)$ has even degree $2j$.
Indeed, the dimension of $M$ satisfies $n=2(\gamma_1+\ldots+\gamma_m)$.

Equivalently, one may demand that 
$$
\int_M \ch_{\gamma_1}(E) \wedge \cdots \wedge \ch_{\gamma_m}(E) \neq 0
$$
for some $\gamma_j\in\N_0$.
Here $\ch(E) = \ch_0(E)+\ch_1(E)+\ldots+\ch_m(E) = \mathrm{rank}(E)+\ch_1(E)+\ldots+\ch_m(E)$ is the Chern character form of $E$.
The Chern numbers and the Chern character numbers can be expressed as linear combinations of each other.

Note that the support of the curvature $R^E$ and hence that of $\c_j(E)$ and $\ch_j(E)$ for $j\ge1$ is contained in the interior of $M$.

\begin{definition} \label{Karea} 
We say that an even-dimensional orientable compact connected Riemannian manifold $M$ with boundary has \emph{infinite $K$-area} if for each $\eps>0$ there exists an admissible $E$ such that $\|R^E\|<\eps$.
\end{definition}

\begin{remark}
This property is independent of the Riemannian metric on $M$.
Changing the metric changes the definition of the norm of $R^E$ but since $M$ is compact, the norms coming from two different metrics are equivalent.

The definition does not require $M$ to have a spin structure.
We only need an orientation so that we can integrate the characteristic forms.
Thus having infinite $K$-area is a property of $M$ as an orientable compact connected manifold.
\end{remark}

Following \cite{G1996} we consider the Adams operations.
Let $E$ be a Hermitian vector bundle over $M$ with connection.
For $k\in\N_0$ there is a virtual bundle $\Psi_kE = \Psi_k^+E - \Psi_k^-E$ with the property 
\begin{equation}
\ch_j(\Psi_k E)=\ch_j(\Psi_k^+ E)-\ch_j(\Psi_k^- E)=k^j\ch_j(E).
\label{eq:chAdams}
\end{equation}
The case $j=0$ shows that the Adams operation $\Psi_k$ preserves the rank.
Both bundles $\Psi_k^+E$ and $\Psi_k^- E$ are universal expressions in tensor products of exterior products of $E$, see \cite{A1989}*{Section~3.2} for details.

For a multi-index $k=(k_1,\ldots,k_m)$ we put 
$$
\Psi_kE := \Psi_{k_1}E\otimes\cdots\otimes\Psi_{k_m}E
$$
and rewrite this virtual bundle as a difference of honest bundles by
\begin{align*}
\Psi_kE 
&=
\bigoplus_{\genfrac{}{}{0pt}{}{\mathrm{even\,\#}}{\mathrm{of}\,-\,\mathrm{'s}}}\Psi_{k_1}^\pm E\otimes\cdots\otimes\Psi_{k_m}^\pm E -\bigoplus_{\genfrac{}{}{0pt}{}{\mathrm{odd\,\#}}{\mathrm{of}\,-\,\mathrm{'s}}}\Psi_{k_1}^\pm E\otimes\cdots\otimes\Psi_{k_m}^\pm E 
=: \Psi^+_kE - \Psi^-_kE .
\end{align*}
Again, $\Psi_k^+E$ and $\Psi_k^- E$ are universal expressions in tensor products of exterior products of $E$.
Hence they inherit natural Hermitian metrics and connections and they are trivial near the boundary if $E$ is so.
Moreover,
\begin{equation}
\|R^{\Psi_k^\pm E}\|\le c_k \|R^E\|
\label{eq:AdamsR}
\end{equation}
where the constant $c_k$ depends only on $k$.

\begin{lemma}\label{lem:nonvanish}
Let $M$ be an oriented compact connected manifold of even dimension $n=2m$ with boundary.
Let $E$ be an admissible bundle of rank $r$.
Let $\omega=1+\omega_1 + \ldots + \omega_m$ be a smooth mixed differential form on $M$ where $\omega_j$ has degree $2j$.

Then there exists $k=(k_1,\ldots,k_m)\in \{0,1,\ldots,m\}^m$ such that 
$$
\int_M \omega\wedge\ch(\Psi_{k}E) \neq r^m \int_M \omega
$$
and 
$$
\|R^{\Psi_k^\pm E}\|\le c(m) \|R^E\|
$$
where $c(m)$ is a constant only depending on $m$.
\end{lemma}

\begin{proof}
For $k=(k_1,\ldots,k_m)\in\N_0^m$ we put 
\begin{align*}
P(k_1,\ldots,k_m)
&:=
\int_M \omega\wedge[\ch(\Psi_{k}E) - r^m] 
=
\int_M \omega\wedge[\ch(\Psi_{k_1}E)\wedge\cdots\wedge\ch(\Psi_{k_m}E) - r^m] .
\end{align*}
Expanding $\omega = 1+\omega_1 + \ldots + \omega_m$ and the Chern characters yields, using \eqref{eq:chAdams},
$$
P(k_1,\ldots,k_m)
=
\sum_{\gamma_1+\ldots+\gamma_m=m} k_1^{\gamma_1}\cdots k_m^{\gamma_m} \int_M \ch_{\gamma_1}(E)\wedge\cdots\wedge\ch_{\gamma_m}(E) + \mbox{l.o.t.}
$$
where l.o.t.\ stands for terms of lower total order in $k_1,\ldots,k_m$.
In particular, $P$ is a polynomial in $k_1,\ldots,k_m$ of total degree at most $m$.

If $P(k_1,\ldots,k_m)=0$ for all $k_i\in \{0,1,\ldots,m\}$ then $P$ would vanish as a polynomial, hence
$$
\int_M \ch_{\gamma_1}(E)\wedge\cdots\wedge\ch_{\gamma_m}(E) = 0
$$
for all $\gamma_i\in\N_0$ with $\gamma_1+\ldots+\gamma_m=m$, contradicting the admissibility of $E$.

Equation~\eqref{eq:AdamsR} implies $\|R^{\Psi^\pm_kE}\|\le c(m) \|R^E\|$ since there are only finitely many possibilies for~$k$.
\end{proof}

\begin{corollary}\label{cor:chmnichtnull}
Let $M$ be an oriented compact connected Riemannian manifold of even dimension $n=2m$ with boundary.
Let $E\to M$ be an admissible bundle of rank $r$.
Then there exists an admissible bundle $F\to M$ such that 
$$
\int_M\ch(F) \neq 0
$$
and 
$$
\|R^F\|\le c(m) \|R^E\|
$$
where $c(m)$ is a constant only depending on $m$.
\end{corollary}

\begin{proof}
Applying Lemma~\ref{lem:nonvanish} with $\omega=1$ yields $k=(k_1,\ldots,k_m)\in \{0,1,\ldots,m\}^m$ such that 
$$
\int_M \ch(\Psi^+_kE) - \int_M \ch(\Psi^-_kE) = \int_M \ch(\Psi_{k}E) \neq 0.
$$
Both $\Psi^+_kE$ and $\Psi^-_kE$ are admissible.
Thus $F=\Psi^+_kE$ or $F=\Psi^-_kE$ does the job.
\end{proof}

\begin{corollary}\label{cor:KareaProduct}
Let $M$ and $N$ be oriented compact connected manifolds, $M$ with boundary and $N$ without boundary.
If $M$ and $N$ have infinite $K$-area then so has $N\times M$.
\end{corollary}

\begin{proof}
We equip $M$ and $N$ with Riemannian metrics and give $N\times M$ the product metric.
Let $\eps>0$.
By Corollary~\ref{cor:chmnichtnull} there exist admissible Hermitian vector bundles $E_M\to M$ and $E_N\to N$ with compatible connections such that
\begin{enumerate}[\myicon]
\item 
$\int_M \ch(E_M) \neq 0$ and $\int_N \ch(E_N) \neq 0$;
\item
$\|R^{E_M}\|<\eps$ and $\|R^{E_N}\|<\eps$.
\end{enumerate}
Then 
$$
\int_{N\times M} \ch(E_N\boxtimes E_M) = \int_{N} \ch(E_N) \cdot \int_{M} \ch(E_M) \neq 0.
$$
Thus $E_N\boxtimes E_M\to N\times M$ is admissible.
Moreover, $\|R^{E_N\boxtimes E_M}\|\le \|R^{E_N}\| +\|R^{E_M}\| < 2\eps$.
\end{proof}

\begin{definition}\label{def:enlargeable}
We say that an $n$-dimensional oriented compact connected Riemannian manifold $M$ with boundary is \emph{area-enlargeable} if for any $\eps>0$ there exists a finite covering $\pi\colon \hat M\to M$ and an $\eps$-area-contracting smooth map $f\colon \hat M\to S^n$ of nonzero degree which is constant on a neighborhood of any connected component of $\partial\hat M$.
Here ``$\eps$-area-contracting'' means that the induced map on $2$-vectors $\Lambda^2df(x)\colon \Lambda^2 T\hat M \to \Lambda^2 TS^n$ is $\eps$-contracting.
\end{definition}

This is an adaptation of the concept of $\Lambda^2$-enlargeability in  \cite{GL}*{Definition 7.1.} to manifolds with boundary.

An even-dimensional area-enlargeable manifold has infinite $K$-area.
Namely, given $\eps>0$ pull back $\Sigma^+S^n$ along an $\eps$-area-contracting map $f$ and obtain $f^*\Sigma^+S^n \to \hat M$.
Now ``integrate over the fibers'' of $\pi$, i.e.\ let $E\to M$ be the bundle with fibers 
$$
E_x = \bigoplus_{y\in\pi^{-1}(x)} (f^*\Sigma^+S^n)_y .
$$

\begin{example}
The $n$-dimensional torus is area-enlargeable and hence has infinite $K$-area if $n$ is even.
If $M$ has infinite $K$-area then $T^k\times M$ has infinite $K$-area as well by Corollary~\ref{cor:KareaProduct} if $k$ is even.
\end{example}

This leads to a stabilized version of infinite $K$-area:

\begin{definition}
We say that an orientable compact connected Riemannian manifold $M$ with boundary has \emph{stably infinite $K$-area} if $T^k\times M$ has infinite $K$-area for some $k$ (and hence for all $k'=k+2\ell$).
\end{definition}

Note that this definition is also meaningful for odd-dimensional $M$.

\begin{lemma}\label{lem:aesiKa}
Let $M$ be an $n$-dimensional oriented compact connected Riemannian manifold $M$ with boundary.
If $M$ is area-enlargeable then it has stably infinite $K$-area.
\end{lemma}

\begin{proof}
For even $n$ this is clear: if $M$ is area-enlargeable it has infinite $K$-area and hence stably infinite $K$-area.
Let $n$ be odd.
We fix a smooth map $f_{n}:S^1\times S^n \to S^{n+1}$ of degree $1$.
Given $\eps,\eps'>0$ we find an $\eps$-area-contracting map $f:\hat M\to S^n$ and an $\eps'$-contracting map $g:\hat S^1\to S^1$ for suitable finite coverings $\hat M \to M$ and $\hat S^1 \to S^1$ such that both $f$ and $g$ have nonzero degrees and $f$ is constant on a neighborhood of any connected component of $\partial\hat M$.
Then $g\times f: \hat S^1\times \hat M \to S^1\times S^n$ is easily checked to be $(\eps+c\eps')$-area-contracting where $c$ depends on $f$ but not on $g$.
Furthermore, it is of nonzero degree and constant on a neighborhood of any connected component of $\partial (  \hat S^1 \times \hat M)$.
Composing with $f_n$ we obtain the map $f_n\circ(g\times f): \hat S^1\times \hat M \to S^{n+1}$ of nonzero degree.
This map is $c_n(\eps+c\eps')$-area-contracting where $c_n$ depends on $f_n$.
Since we can make $c_n(\eps+c\eps')$ arbitrarily small by first choosing $f$ and then $g$, we have that $S^1\times M$ is area-enlargeable.
Thus $S^1\times M$ has infinite $K$-area and hence $M$ has stably infinite $K$-area.
\end{proof}

\begin{definition}\label{def:negligible}
Let $X$ be a compact manifold with boundary (which may be empty).
A subset $Y\subset X$ whose closure is contained in the interior of $X$ is called \emph{$K$-negligible} if there exists a smooth map $f\colon (X,\partial X)\to (X,\partial X)$ of nonzero mapping degree which is constant on a neighborhood of the closure of any connected component of $Y$.
\end{definition}

\begin{example}
\begin{enumerate}[\myicon]
\item
If $Y$ is the union of finitely many disjoint closed balls with smooth boundary, then $Y$ is $K$-negligible.
\item 
Any subset of a $K$-negligible set is $K$-negligible.
\item
Let $X'$ be a closed manifold.
If $Y\subset \mathring{X}$ and $Y'\subset X'$ are $K$-negligible then $Y\times Y'$ is $K$-negligible in $X\times X'$.
\end{enumerate}
\end{example}

\begin{lemma}\label{lem:negligible}
Let $X$ be an $n$-dimensional oriented compact connected Riemannian manifold with boundary and let $Y\subset X$ be a $K$-negligible open subset with smooth boundary.
Put $M:=X\setminus Y$.
Then
\begin{enumerate}[(i)]
\item \label{neglKarea}
if $X$ has infinite $K$-area, so has $M$;
\item \label{neglEnlarge}
if $X$ is area-enlargeable, so is $M$.
\end{enumerate}
\end{lemma}

\begin{proof}
Let $f\colon (X,\partial X)\to (X,\partial X)$ be as in Definition~\ref{def:negligible}.

Ad \ref{neglKarea}.
Given $\eps>0$, let $E\to X$ be an admissible bundle with $\|R^E\|\le\eps$.
Then $\hat E := f^*E|_M \to M$ is admissible with $\|R^E\|\le c\cdot\eps$ where the constant $c$ depends only on $f$.

Ad \ref{neglEnlarge}.
Given $\eps>0$, let $g\colon\hat X\to S^n$ be an $\eps$-area-contracting map as in Definition~\ref{def:enlargeable} for a suitable covering $\pi\colon\hat X\to X$.
Pulling back by $f$ we obtain a commutative diagram
$$
\xymatrix{
f^*\hat X \ar[r]^{\hat f}\ar[d]_{\tilde\pi} & \hat X \ar[d]^\pi \\
X \ar[r]^f & X
}
$$
Put $\hat M := \tilde\pi^{-1}(M) \subset f^*\hat X$.
Then $g\circ\hat f|_{\hat M}\colon \hat M \to S^n$ has nonzero degree, is constant on a neighborhood of any connected component of the boundary of $M$ and is $c\eps$-area-contracting where $c$ depends on $f$.
\end{proof}

\begin{example}
Let $X=T^3$.
The $3$-torus is area-enlargeable.
Let $Y\subset X$ be an open handlebody of genus $g$ with smooth boundary.
We assume that $Y$ is contained in the interior of a small closed $3$-ball inside $X$.
Then $Y$ is $K$-negligible.
Thus $M=X\setminus Y$ is a compact $3$-manifold whose boundary is a surface of genus $g$, see Figure~\ref{fig:threetorus} for the case of genus $1$.
By Lemma~\ref{lem:negligible}, $M$ is also area-enlargeable and hence has stably infinite $K$-area by Lemma~\ref{lem:aesiKa}.

\begin{figure}[!ht]
\includegraphics[width=50mm]{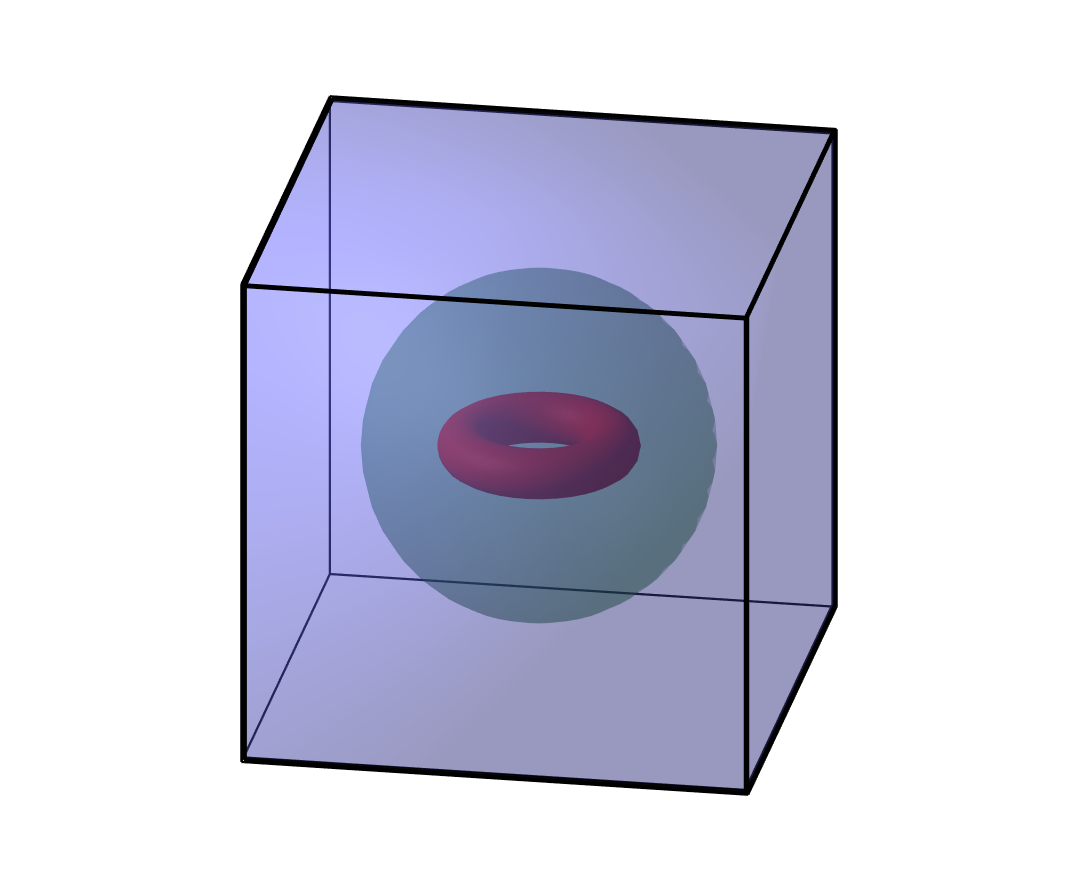}
\caption{Identify opposite sides of cube to obtain $3$-torus and remove the red solid torus}
\label{fig:threetorus} 
\end{figure}
\label{ex:threetorus} 
\end{example}

\subsection{Nonexistence of metrics with positive scalar curvature and mean convex boundary}

The following lemma is probably well known but we include it for the sake of self-containment.
\begin{lemma}\label{lem:conformal}
Let $M$ be an $n$-dimensional compact connected manifold with boundary with $n\ge3$.
Let $g$ be a Riemannian metric on $M$ with $\scal_g\ge0$ and $H_g\ge0$.
If 
\begin{enumerate}[\myicon]
 \item $H_g \not\equiv 0$ or
 \item $\ric_g\not\equiv0$
\end{enumerate}
then $M$ also carries a metric $\tilde g$ with $\scal_{\tilde g}>0$ and $H_{\tilde g}=0$.
\end{lemma}

\begin{proof}
If $\phi\in C^\infty(M;\R)$ is positive on $M$ including the boundary then we can introduce the conformally equivalent metric $\hat g$ by 
$$
\hat g = \phi^\frac{4}{n-2}g.
$$
The scalar curvatures and mean curvatures of the boundary are related by
\begin{align*}
\big(4\tfrac{n-1}{n-2}\Delta_g + \scal_g\big)\phi &= \scal_{\hat g}\cdot\phi^\frac{n+2}{n-2}\quad \text{ on }M, \\
\big(-\tfrac{2}{n-2}\tfrac{\partial}{\partial\nu} + H_g\big)\phi &= H_{\hat g}\cdot\phi^\frac{n}{n-2}\quad \text{ on }\dM.
\end{align*}
Here $\Delta_g=d^*_gd$ is the nonnegative Laplace operator and $\nu$ is the interior unit normal field along the boundary.
The Yamabe operator $4\tfrac{n-1}{n-2}\Delta_g + \scal_g$ together with Robin boundary conditions $\tfrac{\partial\phi}{\partial\nu} - \tfrac{n-2}{2}H_g\phi=0$ yields a self-adjoint operator on $M$.
Hence it has discrete spectrum $\lambda_1 < \lambda_2 < \ldots$ and smooth eigenfunctions.
The first eigenvalue $\lambda_1$ has multiplicity $1$.
By the strong maximum principle its eigenfunctions do not vanish in the interior of $M$, and by Hopf's boundary point lemma they do not vanish on the boundary either.
If we choose $\phi$ to be a positive eigenfunction for $\lambda_1$ then we find
\begin{align}
\big(4\tfrac{n-1}{n-2}\Delta_g + \scal_g\big)\phi &= \lambda_1\phi\quad \text{ on }M, \notag\\
\big(-\tfrac{2}{n-2}\tfrac{\partial}{\partial\nu} + H_g\big)\phi &= 0\quad \text{ on }\dM,
\label{eq:YamabeRobin}
\end{align}
and therefore
\begin{equation*}
\scal_{\hat g} = \lambda_1\phi^\frac{4}{2-n} \quad\text{ and }\quad H_{\hat g}=0.
\end{equation*}
Thus if $\lambda_1>0$ we have $\scal_{\hat g}>0$ and $\tilde g=\hat g$ does the job.

The variational characterization of $\lambda_1$ reads
$$
\lambda_1 = \min_{\phi\in C^\infty(M;\R)\setminus\{0\}}\frac{\int_M(4\frac{n-1}{n-2}|d\phi|_g^2 + \scal_g\phi^2)\dV_g + 2(n-1) \int_\dM H_g\phi^2 \dA_g}{\int_M \phi^2 \dV_g}
$$
where the minimum is attained by the eigenfunctions.
Now if $\scal_g>0$ somewhere or $H_g>0$ somewhere we conclude $\lambda_1>0$ because the principal eigenfunction vanishes nowhere.
The lemma is proved in this case.

It remains to consider the case $\scal_g\equiv0$ and $H_g\equiv0$.
Then we have $\lambda_1=0$ with constant eigenfunctions.
By assumption, there exists a point $p$ in the interior of $M$ at which the Ricci curvature does not vanish, $\ric_g(p)\neq0$.
We choose a cutoff function $\chi\in C^\infty(M;\R)$ with $\chi(p)>0$, $\chi\ge0$ everywhere, and the support of $\chi$ is compact and disjoint from $\dM$.
We consider the $1$-parameter-deformation of $g$ given by
$$
g(t) = g - t\cdot\chi\cdot\ric_g .
$$
The variational formula for the scalar curvature (\cite{Besse}*{Thm.~1.174~(e)}) yields
\begin{align*}
\tfrac{d}{dt}|_{t=0}\scal_{g(t)}
&=
-\Delta_{g}(\tr_g(\chi\cdot\ric_g)) - \div_g\div_g(\chi\cdot\ric_g) + g(\ric_g,\chi\cdot\ric_g) \\
&=
-\Delta_{g}(\chi\cdot\scal_g) - \div_g\div_g(\chi\cdot\ric_g) + \chi|\ric_g|_g^2 \\
&=
- \div_g\div_g(\chi\cdot\ric_g) + \chi|\ric_g|_g^2 .
\end{align*}
Let $\lambda_1(t)$ be the first eigenvalue of \eqref{eq:YamabeRobin} for the metric $g(t)$ and $\phi_t$ the unique positive eigenfunction normalized by 
$$
\int_M \phi_t^2 \dV_{g(t)} = 1.
$$
In particular, $\phi_0^2=\vol(M,g)^{-1}$.
Since $g(t)$ conincides with $g$ near $\dM$ we have $H_{g(t)}\equiv0$ and hence
$$
\lambda_1(t) = \int_M(4\tfrac{n-1}{n-2}|d\phi_t|_{g(t)}^2 + \scal_{g(t)}\phi_t^2)\dV_{g(t)} .
$$
Since $\phi_0$ is constant on $M$, the function $t\mapsto |d\phi_t|_{g(t)}^2$ vanishes to second order at $t=0$.
Because of this and $\scal_g\equiv0$ we find
\begin{align*}
\dot\lambda_1(0)
&=
\int_M (\tfrac{d}{dt}|_{t=0}\scal_{g(t)}) \phi_0^2 \dV_g \\
&=
\vol(M,g)^{-1} \int_M \big(- \div_g\div_g(\chi\cdot\ric_g) + \chi|\ric_g|_g^2\big) \dV_g \\
&=
\vol(M,g)^{-1} \int_M  \chi|\ric_g|_g^2\dV_g \\
&>
0.
\end{align*}
Thus $\lambda_1(t)>0$ for small $t>0$.
Applying the conformal change described above to $g(t)$ yields the desired metric $\tilde g$.
\end{proof}

\begin{theorem}\label{thm:nonexist}
Let $M$ be a compact connected spin manifold with boundary.
Assume $M$ has stably infinite $K$-area.
Then each Riemannian metric $g$ on $M$ with $\scal\ge0$ and $H\ge0$ is Ricci-flat and satisfies $H\equiv0$.
In particular, $M$ does not admit a Riemannian metric with $\scal>0$ and $H\ge0$.

The same holds for $N\times M$ if $N$ is a closed connected spin manifold with nontrivial $\Ahat$-genus.
\end{theorem}

\begin{proof}
It suffices to prove the theorem for $M$ having infinite $K$-area rather than stably infinite $K$-area.
Namely, if the metric on $M$ satisfies $\scal\ge0$ and $H\ge0$, so does the product metric on $T^k\times M$ where $T^k$ is given a flat metric.
If we then know that the metric on $T^k\times M$ is Ricci-flat and satisfies $H\equiv0$, the same holds for the metric on $M$.
Similarly, if $M$ is $2$-dimensional we can replace $M$ by $T^2\times M$.
Thus there is no loss of generality in assuming that $n=\dim(M)\ge4$.

Assume that $M$ has infinite $K$-area and write $n=2m$ with $m\ge2$.
Let $g$ be a metric on $M$ with $\scal\ge0$ and $H\ge0$.
If $g$ is not Ricci-flat or $H\not\equiv0$, then by Lemma~\ref{lem:conformal} we may assume without loss of generality that $g$ satisfies $\scal>0$ and $H\ge0$.
We will derive a contradiction from this.

For any admissible $E$, the virtual bundle $\Psi_kE$ has the same rank as $E$ which we denote by~$r$.
Hence, for any $k=(k_1,\ldots,k_m)\in\N_0^m$, the virtual bundle $\Psi_k E$ has rank $r^m$ and thus $\Psi_k E- E_0^{r^m}$ has rank $0$ where $E_0^{r^m}$ denotes the trivial bundle of rank $r^m$.
We equip $E_0^{r^m}$ with the trivial connection.
Now we rewrite the virtual bundle $\Psi_kE- E_0^{r^m}$ as a difference of honest bundles by
\begin{align*}
\Psi_kE- E_0^{r^m}
&=
\Psi_k^+E - (\Psi_k^-E \oplus E_0^{r^m})
=: F^+_k - F^-_k .
\end{align*}
By \eqref{eq:AdamsR} there is a constant $c(m)$ depending only on $m$ such that 
$$
\|R^{F^\pm_k}\| \leq c(m) \|R^E\|
$$
for all $k\in\{0,\ldots,m\}^m$.

Recall that the metric on $M$ satisfies $\scal>0$.
Since $M$ has infinite $K$-area we can find an admissible $E$ with $2n(n-1)c(m)\|R^E\|<\min\scal$.
We choose $k=(k_1,\ldots,k_m)\in\{0,1,\ldots,m\}^m$ as in Lemma~\ref{lem:nonvanish} with $\omega=\Ahat(M)$.
We consider the APS-indices for the twist bundles $F^+_k$ and $F^-_k$.
Both bundles are trivial on a neighborhood of the boundary with trivial connection and have the same rank.
Thus the boundary terms in formula \eqref{eq:APS} coincide for the twist bundles $F^+_k$ and $F^-_k$.
We find
\begin{align*}
\ind\Big(D_{F^+_k}^{+,\mathrm{APS}}\Big) - \ind\Big(D_{F^-_k}^{+,\mathrm{APS}}\Big)
&=
 \int_M \Ahat(M) \wedge [\ch(F^+_k) - \ch(F^-_k)] \\
&=
\int_M \Ahat(M)\wedge[\ch(\Psi_kE) - r^m] 
\neq
0.
\end{align*}
On the other hand, we have $2n(n-1)\|R^{F^\pm_k}\|<\min\scal$.
Proposition~\ref{prop:vanish} implies that 
$$
\ind\Big(D_{F^+_k}^{+,\mathrm{APS}}\Big) = \ind\Big(D_{F^-_k}^{+,\mathrm{APS}}\Big) = 0,
$$
which is a contradiction.
Thus $M$ does not carry a metric with $\scal>0$ and $H\ge0$.

Now assume that $N\times M$ carries such a metric $g$.
Choose auxiliary metrics $g_M$ and $g_N$ on $M$ and $N$, respectively.
We obtain a second metric $g':=g_N\oplus g_M$ on $N\times M$.
There is a constant $c''>0$ such that $\|R^E\|_g \le c''\|R^E\|_{g'}$ for all Hermitian bundles $E$ with connection on $N\times M$.
Put $\tilde n:=\dim(N)$.

We choose an admissible $E$ on $M$ with $2(n+\tilde n)(n+\tilde n-1)c_{m}'c''\|R^E\|_{g_M}<\min_{N\times M}\scal_g$.
Again, we choose $k=(k_1,\ldots,k_m)\in\{0,1,\ldots,m\}^m$ as in Lemma~\ref{lem:nonvanish} with $\omega=\Ahat(M)$.
We pull back $E$ to $N\times M$ along the projection $\pi_M\colon N\times M\to M$ onto the second factor.
Then 
\begin{align*}
2(n+\tilde n)(n+\tilde n-1)c(m)\|R^{\pi^*_ME}\|_{g}
&\le 
2(n+\tilde n)(n+\tilde n-1)c(m)c''\|R^{\pi^*_ME}\|_{g'} \\
&= 
2(n+\tilde n)(n+\tilde n-1)c(m)c''\|R^{E}\|_{g_M} \\
&<
\min_{N\times M}\scal 
\end{align*}
and hence
$$
2(n+\tilde n)(n+\tilde n-1)\|R^{\pi^*_MF_k^\pm}\|_{g}<\min_{N\times M}\scal .
$$
Proposition~\ref{prop:vanish} applied to $(N\times M,g)$ yields 
\begin{equation}
\ind\Big(D_{\pi^*_MF^\pm_k}^{+,\mathrm{APS},g}\Big)  = 0.
\label{eq:indexwieder0}
\end{equation}
Here the additional upper index $g$ indicates that we are taking the Dirac operator with respect to the metric $g$.
Note that the index may depend on the choice of metric since the boundary metric and hence the APS-boundary conditions depend on it. 
In order to control this, consider the transgression form $T\Ahat(N\times M,g,g')$ for the $\Ahat$-form with respect to the two metrics $g$ and $g'$, i.e.\
$$
\Ahat(N\times M,g) - \Ahat(N\times M,g') = dT\Ahat(N\times M,g,g').
$$
We compute
\begin{align*}
\ind\Big(D_{F^+_k}^{+,\mathrm{APS},g}\Big) &- \ind\Big(D_{F^-_k}^{+,\mathrm{APS},g}\Big) \\
&=
 \int_{N\times M} \Ahat(N\times M,g) \wedge [\ch(\pi^*_MF^+_k) - \ch(\pi^*_MF^-_k)] \\
&=
 \int_{N\times M} \big(\Ahat(N\times M,g')+dT\Ahat(N\times M,g,g')\big) \wedge [\ch(\pi^*_MF^+_k) - \ch(\pi^*_MF^-_k)] \\
&=
\int_{N\times M} \pi^*_N\Ahat(N,g_N)\wedge\pi^*_M\Ahat(M,g_M)\wedge\pi^*_M[\ch(F^+_k) - \ch(F^-_k)] \\
&\quad
+ \int_{N\times M} d\big(T\Ahat(N\times M,g,g') \wedge [\ch(\pi^*_MF^+_k) - \ch(\pi^*_MF^-_k)]\big) \\
&=
\int_{N} \Ahat(N,g_N)\cdot\int_M\Ahat(M,g_M)\wedge[\ch(F^+_k) - \ch(F^-_k)] \\
&\quad
+ \int_{N\times\dM } T\Ahat(N\times M,g,g') \wedge \pi^*_M[\ch(F^+_k) - \ch(F^-_k)] .
\end{align*}
The boundary integral vanishes because $\ch(F^+_k) - \ch(F^-_k)$ vanishes on a neighborhood of $\dM$.
Hence
$$
\ind\Big(D_{F^+_k}^{+,\mathrm{APS},g}\Big) - \ind\Big(D_{F^-_k}^{+,\mathrm{APS},g}\Big)
=
\int_{N} \Ahat(N,g_N)\cdot\int_M \Ahat(M)\wedge[\ch(\Psi_kE) - r^m] 
\neq 
0.
$$
This contradicts \eqref{eq:indexwieder0}.
\end{proof}

The following example shows that the conclusion in Theorem~\ref{thm:nonexist} is optimal.

\begin{example}
Let $M=S^1(1)\times I$ where $I=[-1,1]$ and $S^1(1)$ is the circle of length $1$.
Let $f_1\colon I\times I\to S^2$ be a map of degree $1$ which maps a disk of radius $\frac12$ centered at $(0,0)$ onto the sphere and its complement onto the north pole.
This map descends to a smooth degree-$1$ map $M\to S^2$.
It is $c$-contracting for some $c>0$ and hence $c^2$-area-contracting.

For $k\in\N$ consider the map $\phi_k\colon [-k,k]\times I \to I\times I$ given by $\phi_k(s,t)=(s/k,t)$.
The map $f_k = f_1\circ\phi_k$ descends to a smooth degree-$1$ map $S^1(2k)\times I \to S^2$ which is $\frac{c^2}{k}$-area contracting.
Clearly, $S^1(2k)\times I$ is a $2k$-fold Riemannian covering space of $M$.
Thus $M$ is area-enlargeable and hence has infinite $K$-area (see Figure~\ref{fig:2}).
\begin{figure}[!ht]
\begin{annotate}
{\includegraphics[scale=0.5]{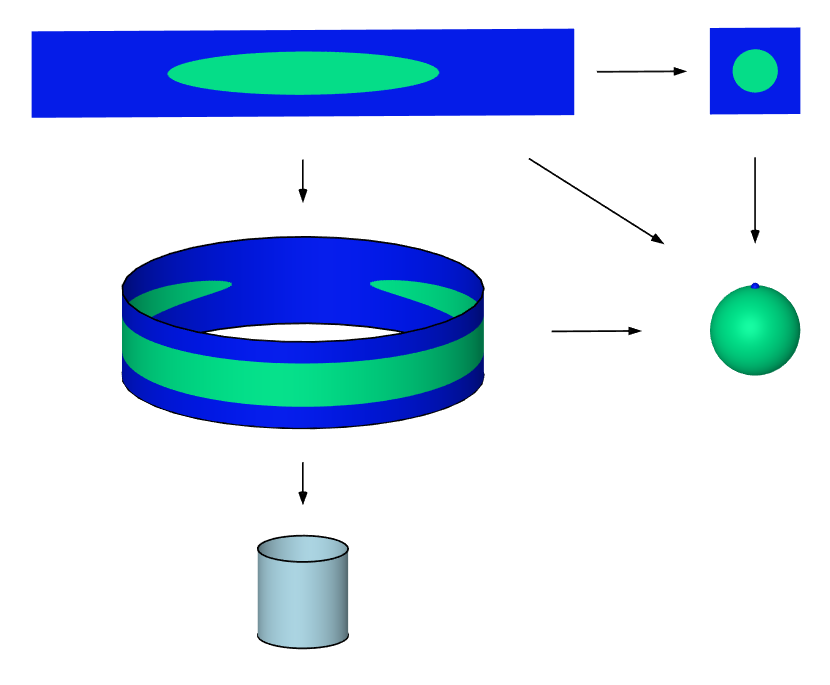}}{0.5}
\note{-9.4,0.2}{\scalebox{1.4}{$S^1(2k)\times I$}}
\note{-4.6,-6}{\scalebox{1.4}{$M$}}
\note{10,0.5}{\scalebox{1.4}{$S^2$}}
\note{8.6,3.5}{\scalebox{1.4}{$f_1$}}
\note{5.5,6.8}{\scalebox{1.4}{$\phi_k$}}
\note{4.9,3.7}{\scalebox{1.4}{$f_k$}}
\end{annotate}

\caption{$M=S^1\times I$ is area-enlargeable}
\label{fig:2} 
\end{figure}

Note that $M$ is flat with totally geodesic boundary.
Now let $N$ be a $K3$-surface equipped with a Ricci-flat metric.
Then $N$ is spin and has $\Ahat$-genus $2$.
Thus $N\times M$ is as in Theorem~\ref{thm:nonexist} and has a Ricci-flat metric with $H\equiv0$.
This metric is not flat though.
\end{example}

\section{Deformations of the metric} \label{spaces} 

Let $M$ be a smooth manifold of dimension $n \geq 2$ with compact boundary $\dM$. 
Given a Riemannian metric $g$ on $M$ we denote by 
\begin{enumerate}[\myicon] 
   \item $\scal_g \colon  M \to \R$ the scalar curvature of $g$, 
   \item $g|_{\dM }\in C^\infty(\dM;(T^*M\otimes T^*M)|_{\dM})$ the restriction of $g$ to $\dM $, 
   \item $g_0\in C^\infty(\dM;T^*\dM\otimes T^*\dM)$ the metric induced on $\dM$,
   \item $\II_g$ the second fundamental form of $\dM \subset M$ with respect to the interior unit normal, 
   \item $H_g  =  \frac{1}{n-1}\tr_g (\II_g)\colon  \dM \to \R$ the mean curvature of $\dM$.
\end{enumerate}  
We denote by $\RR(M)$ the space of smooth Riemannian metrics  on $M$, equipped with the weak $C^{\infty}$-topology. 
Let $\sigma \colon  M \to \R$ be a continuous function which is unchanged throughout this section. 
We put $\RR_{> \sigma}(M) := \{ g\in\RR(M) \mid \scal_g > \sigma \}$. 

The normal exponential map with respect to $g$ along $\dM$ yields a diffeomorphism 
\begin{equation}
[0,\eps) \times \dM  \xrightarrow{\approx} U^g_\eps
\label{eq:normalexp}
\end{equation}
onto the open $\eps$-neighborhood $U^g_\eps$ of $\dM $ with respect to $g$, for $\eps>0$ sufficiently small.
In particular, it induces a smooth structure on the double $\D M = M \cup_{\dM } M$.

We will construct certain deformations of a given metric.
These deformations will be supported near the boundary, in a neighborhood $U^g_\eps$ for small $\eps$.
It will be technically useful to first deform the metric into a standard form near the boundary.
In order to make this more precise, recall that, by the generalized Gauss lemma, any metric $g$ takes the form 
\begin{equation}
     g = dt^2 + g_t , 
\label{eq:metricnearboundary}
\end{equation}
near the boundary.
Here, we use the identification in \eqref{eq:normalexp}, $t$ is the canonical coordinate in $[0,\eps)$ (in other words, the distance from $\dM $), and $(g_t)_{t \in I}$ is a smooth family of Riemannian metrics on $\dM $. 

For any smooth family $(h_t)_{t \in I}$ of $(2,0)$-tensor fields on $\dM$ we define new families $\dot h_t$, $\ddot h_t$, and $h^{(\ell)}_t$ of smooth $(2,0)$-tensor fields  by
\begin{align*} \label{dots} 
      \dot h_t (X,Y) & := \frac{d}{dt} (h_t(X,Y)), \\
      \ddot h_t(X,Y) & := \frac{d^2}{dt^2} (h_t(X,Y)), \\
      h^{(\ell)}_t(X,Y) & := \frac{d^\ell}{dt^\ell} (h_t(X,Y)) .
\end{align*} 
For $0\le t < \eps$
\begin{equation}
   \II_t =  - \tfrac{1}{2} \dot g_t
   \label{eq:II}
\end{equation}
is the second fundamental form (w.r.t.\ the gradient field of $t$) of the hypersurface $N_t \subset M$ at distance $t$ from $\dM $. 
We also consider the Weingarten map $W_t \colon  TN_t \to TN_t$, which is uniquely determined  by the equation
\[
   \langle W_t(X), Y \rangle_{g_t}  =  \II_t(X,Y) = - \tfrac{1}{2} \dot g_t(X, Y) \, . 
\]
Furthermore, 
\[
   H_t = \tfrac{1}{n-1}\tr (W_t) = - \tfrac{1}{2(n-1)} \tr_{g_t} ( \dot g_t)  
\]
is the mean curvature of $N_t$.
Note that in this notation we have $\II_g = \II_0$ and $H_g =  H_0$. 
The scalar curvature of $(M , g)$  is given by 
\begin{equation} \label{scal} 
   \scal_g = \scal_{g_t} + 3\,  \tr(W_t^2)  - \tr(W_t)^2 - \tr_{g_t} ( \ddot g_t) . 
\end{equation}  
See \cite{BGM}*{Prop.~4.1} for details.

\begin{definition}  \label{C-normal} 
Let $C \in \R$. 
A metric $g$ on $M$ is called \emph{$C$-normal} if the $g_t$ in \eqref{eq:metricnearboundary} are given by
\begin{align*}
g_t 
&= g_0 + t\cdot\dot g_0   -   Ct^2 \cdot  g_0  
= g_0 - 2t\cdot\II_{g}   -   Ct^2 \cdot  g_0 .
\end{align*}
\end{definition}  

\begin{remark} \label{remnormal} 
Near the boundary, a $C$-normal metric is uniquely determined by the first and second fundamental form of the boundary and the constant $C$.

For sufficiently large $C$ the scalar curvature of a $C$-normal metric $g$ satisfies $\scal_g > \sigma$ near the boundary.
More precisely,  let $g_0$ be a Riemannian metric and $h$ a symmetric $(2,0)$-tensor field on $\dM $. 
Let $W_0 \colon  T\dM  \to T\dM $ be defined by $\langle W_0(X), Y \rangle_{g_0} =  h(X,Y) $ and set 
\[
     C_0 = C_0(g_0,h, \sigma|_{\dM}) := - \frac{1}{2(n-1)} \max_{\dM } \left(   \scal_{g_0}  - \sigma + 3\, \tr(W_0^2) - \tr(W_0)^2  \right) .
\]
Then the condition $C > C_0$ is equivalent to the scalar curvature of 
\[
    g = dt^2 + g_0 - 2 t \cdot h  - Ct^2  \cdot g_0 
\]
being greater than $\sigma$ along  $\dM $ (and hence in a neighborhood of $\dM $).
This follows from $\tr_{g_0}(\ddot g_0)=-2C(n-1)$ and formula \eqref{scal}.
\end{remark}

\begin{proposition} \label{step1}
Let $K$ be a compact Hausdorff space and let $g \colon  K \to \RR_{> \sigma}( M )$ be a continuous family of metrics of scalar curvature greater than $\sigma$.

Then there exists a constant $C_0$ such that for each $C\ge C_0$ and each neighborhood $\U$ of $\dM$ there exists a continuous map $f \colon  K \times [0,1] \to \RR_{> \sigma} (M)$ such that the following holds  for all $\xi \in K$ and $s \in [0,1]$: 
\begin{enumerate}[(a)]  
   \item $f(\xi,0) = g(\xi)$; 
   \item  \label{cnorm} $f(\xi,1)$ is $C$-normal; 
   \item if $g(\xi)$ is $\widetilde C$-normal then $f(\xi,s)$ is $\big((1-s)\widetilde C +sC\big)$-normal;
   \item \label{initial} $f(\xi,s)|_{\dM } = g(\xi)|_{\dM }$, in particular $f(\xi,s)_0 = g(\xi)_0$, and $\II_{f(\xi,s)} = \II_{g(\xi)}$;
   \item $\ddot f(\xi,s)_0 = (1-s) \ddot g(\xi)_0 -2 s  C g(\xi)_0$; 
   \item \label{higherdiffs} $f(\xi,s)^{(\ell)}_0 = (1-s)g(\xi)^{(\ell)}_0$ for all $\ell\ge3$;
   \item $f(\xi,s)=g(\xi)$ on $M\setminus \U$.
\end{enumerate} 
\end{proposition} 
 
\begin{proof} For $\xi \in K$ we consider the Taylor expansion of the smooth map $t \mapsto g(\xi)_t$ at $t = 0$, 
\[
     g(\xi)_t = g(\xi)_0 + \dot g(\xi)_0  \cdot t + \tfrac{1}{2} \ddot g(\xi)_0 \cdot t^2 +  R(\xi)_t .  
\]
Here  $(R(\xi)_t)_{t}$ is a smooth family of symmetric $(2,0)$-tensor fields on $\dM $ which depends continuously on $\xi$ and satisfies 
\begin{equation} \label{remainder} 
     R(\xi)_0 = \dot R(\xi)_0 = \ddot R(\xi)_0 = 0 . 
 \end{equation} 
Set
\begin{equation} \label{defC} 
    C_0 :=  \tfrac{1}{2(n-1)} \max_{\xi \in K}   \| \tr_{g(\xi)_0} ( \ddot g(\xi)_0 )\|_{C^0(\dM )}
\end{equation} 
and let $C\ge C_0$.
Put
\begin{equation} \label{defF} 
  F (\xi,s) := g(\xi)  - s \left( \left( \tfrac{1}{2} \ddot g(\xi)_0  +  C   \cdot g(\xi)_0 \right) \cdot  t^2  + R(\xi)_t \right) . 
 \end{equation} 
Each $F(\xi,s)$ is a Riemannian metric on a neighborhood of $\dM $. 
The neighborhood can be chosen independently of $\xi$ and $s$ because $F$ depends continuously on $(\xi,s)$ and $K$ is compact.

According to  \eqref{scal}, \eqref{remainder}, and \eqref{defC}, $F(\xi,s)$ satisfies along the boundary
 \[
    \scal_{F(\xi,s)} \big|_{ \dM } 
    = 
    \scal_{g(\xi)} \big|_{\dM }  +  s\cdot \left( \tr_{g(\xi)_0} (\ddot g(\xi)_0) + 2C\cdot (n-1)\right) 
    \geq 
    \scal_{g(\xi)} \big|_{\dM } 
    > 
    \sigma|_{\dM }  . 
\]
Hence all $F(\xi,s)$ have scalar curvature greater than $\sigma$ on a common neighborhood $U$ of $\dM $.
By construction, we have for all $\xi \in K$ and $s\in[0,1]$:
\begin{enumerate}[\myicon]
\item
$F(\xi,s)|_{\dM } = g(\xi)|_{\dM }$;
\item
$\dot F(\xi,s)_0 = \dot g(\xi)_0$;
\item
$\ddot F(\xi,s)_0 = (1-s)\ddot g(\xi)_0 -2sC g(\xi)_0$;
\item
$F(\xi,s)^{(\ell)}_0 = (1-s)g(\xi)^{(\ell)}_0$ for all $\ell\ge3$.
\end{enumerate}
 
By the family version of the flexibility lemma (\cite{BH}*{Addendum~3.4}) applied for the second-order open partial differential relation on the space of Riemannian metrics on $M$, defined by the condition that scalar curvature is larger than $\sigma$, $F$ can be replaced by a deformation $f$ of $g$ defined on all of $M$ which coincides with $F$ near $\dM $ and always solves the partial differential relation\footnote{By attaching a small cylinder to $\partial M$ we may pass to a smooth manifold without boundary, as required in \cite{BH}.}. 

Specifically, we obtain an open neighborhood $\dM   \subset U_0 \subset U$ and a continuous map $f \colon  K \times [0,1] \to C^{\infty}  (M , T^*M\otimes T^*M)$ such that for all $\xi \in K$ and $s \in [0,1]$ we have 
\begin{enumerate}[(i)] 
   \item $f(\xi,s) \in \RR_{> \sigma} (M)$, 
   \item $f(\xi,s)|_{U_0 } = F(\xi,s)|_{U_0}$, 
   \item $f(\xi,s)|_{M \setminus U} = g(\xi)|_{M \setminus U}$.
\end{enumerate} 
The neighborhood $U$ in the flexibility lemma may be chosen arbitrarily small; 
in particular, we may assume $U\subset\U$.
Then this $f$ does the job. 
\end{proof} 

\begin{lemma}\label{comptr}  
Let $V$ be a finite dimensional real vector space and let $g_1$ and $g_0$ be two Euclidean scalar products on $V$ such that $\|g_1-g_0\|_{g_0}\le\frac12$.
Then 
\[
    | \tr_{g_1} (h) - \tr_{g_0}(h) | \leq 2 \cdot \| g_1 - g _0\|_{g_0} \cdot \| h\|_{g_0}
\]
holds for all symmetric bilinear forms $h$ on $V$.
Here $\| \cdot \|_{g_0}$ denotes the norm on the space of symmetric bilinear forms induced by $g_0$.
\end{lemma}

\begin{proof}
Let $e_1,\ldots,e_n$ be a basis of $V$, orthonormal for $g_0$ and diagonalizing for $g_1$.
Hence $g_1(e_j,e_k) = \kappa_j \delta_{jk}$.
Writing $\kappa_j=1+\eta_j$ we find
$$
\frac14
\ge
\|g_1-g_0\|_{g_0}^2
=
\sum_{j=1}^n \eta_j^2
$$
and thus $|\eta_j|\le\frac12$ for each $j$.
Observe that for $|\eta|<1$ we have
$$
|1 - (1+\eta)| = |\eta| \le \frac{1+\eta}{1-|\eta|}|\eta| 
$$
and therefore
$$
|(1+\eta)^{-1} - 1| \le (1-|\eta|)^{-1} |\eta| .
$$
We compute, writing $h(e_j,e_k)=h_{jk}$ and using the Cauchy-Schwarz inequality,
\begin{align*}
|\tr_{g_1} (h) - \tr_{g_0}(h)|^2
&=
\Big|\sum_{j=1}^n (\kappa_j^{-1} h_{jj} - h_{jj})\Big|^2 \\
&=
\Big|\sum_{j=1}^n ((1+\eta_j)^{-1} -1)h_{jj}\Big|^2 \\
&\le
\Big(\sum_{j=1}^n ((1+\eta_j)^{-1} -1)^2\Big) \cdot \sum_{j=1}^n h_{jj}^2 \\
&\le
\Big(\sum_{j=1}^n ((1-|\eta_j|)^{-2} \eta_j^2\Big) \cdot \sum_{j=1}^n h_{jj}^2 \\
&\le
4\Big(\sum_{j=1}^n \eta_j^2\Big) \cdot \sum_{j=1}^n h_{jj}^2 \\
&\le
4 \| g_1 - g _0\|_{g_0}^2 \cdot \| h\|_{g_0}^2.
\qedhere
\end{align*}
\end{proof}

\begin{lemma} \label{chi} 
There exists a constant $c_0>0$ such that for each $0<\delta \le\frac12$ there exists a smooth function $\chi_\delta  \colon  [0, \infty)   \to \R$ with 
\begin{enumerate}[\myicon] 
   \item $\chi_\delta(t)=t$ for $t$ near $0$, $\chi_\delta(t) = 0  $ for $t \geq \sqrt{\delta}$ and $0\leq \chi_\delta(t)\le \frac{\delta}{2} $ for all $t$,
   \item $|\dot\chi_\delta(t)| \le c_0$ for all $t$,
   \item $ -  \frac{2}{\delta} \leq \ddot\chi_\delta(t) \leq 0$ for all $t\in[0,\delta]$ and $|\ddot\chi_\delta(t)|\le c_0$ for all $t\in[\delta,\sqrt{\delta}]$.
\end{enumerate} 
\end{lemma} 

\begin{proof}
The $C^2$-function $\tilde\phi\colon [0,\infty) \to \R$ defined by
\[
\tilde\phi(t) = 
\begin{cases} 
t-\frac12  &\textrm{for } 0 \leq t \leq \frac{1}{10}  , \\ 
\frac{1}{10240} \, (10 \, t + 7) (10 \, t - 9)^{3} &\textrm{for } \frac{1}{10} \leq t \leq \frac{9}{10} ,   \\ 
0  &\textrm{for }  t\geq \frac{9}{10} ,  
\end{cases} 
\]
satisfies $-\frac12\le\tilde\phi\le 0$ and $-\frac{15}{8}\le\ddot{\tilde\phi}\le0$.
In particular, $\tilde\phi$ is concave.
Now let $\phi_1$ be a smooth concave approximation of $\tilde\phi$ which coincides with $\tilde\phi$ near $0$ and on $[\frac{19}{20},\infty)$ and such that $\|\phi_1-\tilde\phi\|_{C^2([0,1])}\le\frac{1}{8}$.
Then $-\frac12\le\phi_1\le 0$ and $-2\le\ddot\phi_1\le0$.
Moreover, $\ddot\phi_1\le0$ implies $1=\dot{\phi_1}(0)\ge\dot{\phi_1}(t)\ge\dot{\phi_1}(1)=0$.

We put $\phi_\delta(t):=\delta\phi_1(t/\delta)$ and find
\begin{enumerate}[\myicon] 
\item $\phi_\delta(t)=t-\frac\delta2$ for $t$ near $0$;
\item $\phi_\delta(t)=0$ for $t\ge\frac{19}{20}\delta$;
\item $-\frac{\delta}2\le\phi_\delta\le 0$ everywhere;
\item $0\le\dot{\phi_\delta}\le1$ everywhere;
\item $-\frac{2}{\delta}\le\ddot\phi_\delta\le0$ everywhere.
\end{enumerate}
Next, let $\psi_1\colon [0,\infty)\to\R$ be a smooth function with $\psi_1(t)=\frac12$ for $t\in[0,\frac{19}{20}]$, $\psi_1(t)=0$ for $t\ge 1$, and $0\le\psi_1\le\frac12$ everywhere.
For $0<\delta<1$ put $\psi_\delta(t) := \delta\psi_1(t/\sqrt{\delta})$.
Then
\begin{enumerate}[\myicon] 
\item $\psi_\delta(t)=\frac\delta2$ for $t\in[0,\frac{19}{20}\sqrt{\delta}]$;
\item $\psi_\delta(t)=0$ for $t\ge\sqrt{\delta}$;
\item $0\le\psi_\delta\le\frac\delta2$ everywhere;
\item $|\dot\psi_\delta|\le\sqrt{\delta}\|\dot\psi_1\|_{C^0([0,1])}$ everywhere;
\item $|\ddot\psi_\delta|\le \|\ddot\psi_1\|_{C^0([0,1])}$ everywhere.
\end{enumerate}
Now $\chi_\delta:=\phi_\delta+\psi_\delta$ does the job.
Note in particular, that for $t\in [0,\delta]\subset [0,\frac{19}{20}\sqrt{\delta}]$ we have $\ddot\chi_\delta=\ddot\phi_\delta\in[-\frac2\delta,0]$.

The functions $\phi_1$, $\psi_1$, and $\chi_\delta$ are illustrated in Figure~\ref{fig:3}. 
\end{proof}

\begin{figure}[!ht]
\begin{annotate}
{\includegraphics[scale=0.6]{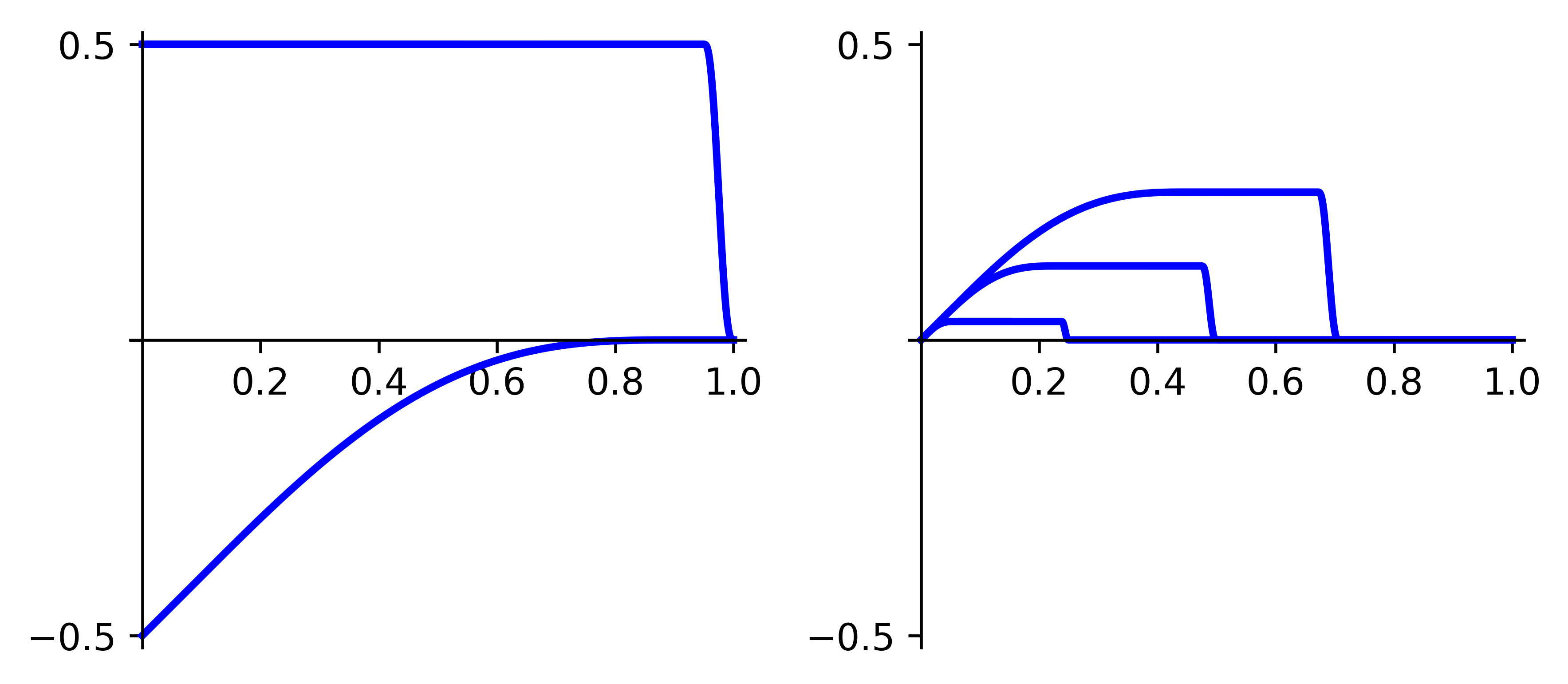}}{0.6}
\note{-3.5,2.4}{\scalebox{1.4}{$\psi_1$}}
\note{-3.5,-1.1}{\scalebox{1.4}{$\phi_1$}}
\note{4,1.8}{\scalebox{1}{$\delta=\nicefrac{1}{2}$}}
\note{3.5,1}{\scalebox{1}{$\delta=\nicefrac{1}{4}$}}
\note{2.8,0.4}{\scalebox{1}{$\delta=\nicefrac{1}{16}$}}
\end{annotate}

\caption{The functions $\phi_1$, $\psi_1$, and $\chi_\delta$}
\label{fig:3} 
\end{figure}

\begin{proposition} \label{schwer} 
Let $K$ be a compact Hausdorff space.
Let $g_0 \colon  K \to C^{\infty} (\dM ; T^*\dM  \otimes T^*\dM )$ be a continuous family of Riemannian metrics on $\dM $ and let $h,k \colon  K \to  C^{\infty} (\dM ; T^*\dM  \otimes T^*\dM )$ be continuous families of symmetric $(2,0)$-tensor fields satisfying $\tr_{g_0}(h) \geq \tr_{g_0}(k)$. 

Then there exists a constant $C_0 = C_0(g_0, h, k) > 0$ such that
\begin{enumerate}[\myicon]
\item
for every continuous family 
$$
g \colon  K \to \RR_{> \sigma} (M)
$$ 
of $C$-normal metrics of scalar curvature greater than $\sigma$ with $C\ge C_0$, $g(\xi)_0=g_0(\xi)$ and $\II_{g(\xi)}=h(\xi)$ for all $\xi\in K$ and
\item
for each neighborhood $\U$ of $\dM$
\end{enumerate} 
there exists a continuous map 
$$
f \colon  K \times [0,1] \to \RR_{> \sigma} (M)
$$ 
so that the following holds for all $\xi \in K$ and $s \in [0,1]$: 
\begin{enumerate}[(a)] 
\item\label{eins}  
$f(\xi,0) = g(\xi)$; 
\item 
$f(\xi,s)$ is $C$-normal;
\item
$f(\xi,s)|_{\dM }=g(\xi)|_{\dM }$, in particular $f(\xi,s)_0 = g_0(\xi)$;
\item\label{IIdeform} 
$\II_{f(\xi,s)} = (1-s)\II_{g(\xi)}+sk(\xi)$;
\item \label{nearbdy}
$f(\xi,s)=g(\xi)$ on $M\setminus \U$.
\end{enumerate} 
\end{proposition}

\begin{proof} 
Let $C \in \R$, $C > 0$, let $g \colon  K \to \RR_{> \sigma}(M)$ be a continuous family of $C$-normal metrics of scalar curvature greater than $\sigma$ such that
\begin{equation} \label{ass} 
   g(\xi) = dt^2 + g_0(\xi) -2t\cdot h(\xi) -   Ct^2 \cdot g_0(\xi) 
\end{equation} 
holds on $\U$ for all $\xi\in K$, after possibly shrinking $\U$.
Choose $\eps_0>0$ such that $U^{g(\xi)}_{\eps_0}\subset \U$ for $\xi\in K$.

We will use the notation $\lesssim$ to mean that the LHS is bounded by the RHS times a positive constant which only depends on the families $g_0$, $h$ and $k$ (and thus on $\dM $).
However, the constant is independent of $\delta$, $t$, $\xi$, $s$, $C$. 
Similarly, we say that a statement holds for sufficiently large $C$ if it holds for all $C$ larger than a constant which depends only on on the families $g_0$, $h$ and $k$.
Finally, we say that a statement holds for sufficiently small $\delta$ if it holds for all $\delta\in(0,\delta_0]$ where $\delta_0$ is a positive constant which depends only on $g_0$, $h$, and $C$.

Let $0<\delta<\min\{\frac12,\eps_0^2,C^{-2}\}$.
Take $\chi_\delta$ as in Lemma~\ref{chi} and consider 
\begin{gather}
f^{\delta} \colon  K \times [0,1] \to C^{\infty}(M;T^*M \otimes T^*M) ,\notag\\
f^{\delta}(\xi,s) = 
\begin{cases}
dt^2 + (1-Ct^2)\cdot g_0(\xi) -2t\cdot h(\xi) + 2s \chi_\delta(t) \cdot (h(\xi)-k(\xi))   & \mbox{ for }t\leq\sqrt{\delta} , \\
g(\xi) &\mbox{ for }t\geq\sqrt{\delta} . 
\end{cases}
\label{def:deform}
\end{gather}
This defines a family of smooth $(2,0)$-tensor fields on $M$.
For $\delta$ sufficiently small, the $f^\delta(\xi,s)$ are positive definite and hence Riemannian metrics on $M$.
On $M\setminus U^{f^\delta(\xi,s)}_{\sqrt{\delta}}$ these metrics have scalar curvature greater than $\sigma$ because they coincide with $g(\xi)$.
Since $\chi_\delta(t) = t$ for small $t$ we find near $\dM $:
\begin{align*}
f^{\delta}(\xi,s) 
&=
dt^2 + (1-Ct^2)\cdot g_0(\xi) -2t \big((1-s) \cdot h(\xi) + s\cdot k(\xi)\big) .
\end{align*}
Hence $f^{\delta}(\xi,s)|_{\dM }=g(\xi)|_{\dM }$ and $\II_{f^{\delta}(\xi,s)} = (1-s)\II_{g(\xi)} + sk(\xi)$.
Provided every $f^\delta(\xi,s)$ has scalar curvature greater than $\sigma$, we have verified properties \ref{eins} -- \ref{nearbdy} for $f=f^\delta$.
It remains to show that $f^\delta(\xi,s)$ has scalar curvature greater than $\sigma$ if $\delta$ is sufficiently small.
We only need to check this on $U^{f^\delta(\xi,s)}_{\sqrt{\delta}}$.

Fix $\xi \in K$ and $s \in [0,1]$. 
In the following we write  $\gamma := f^{\delta}(\xi,s)$ for simplicity and use the splitting $\gamma = dt^2 + \gamma_t$ on $U^\gamma_\eps$.
Thus 
\begin{equation}
\gamma_t = (1-Ct^2)\cdot g_0(\xi) -2t\cdot h(\xi) + 2s \chi_\delta(t) \cdot (h(\xi)-k(\xi)) .
\label{deform}
\end{equation}
In particular, we have $\gamma_0=g(\xi)_0=g_0(\xi)$.

Denote the second fundamental forms and the Weingarten maps of the level sets of the distance function $t$ from the boundary (w.r.t.\ $\gamma$) by $\II_t$ and $W_t$, respectively.
Since $W_t$ is a symmetric endomorphism field we have $\tr(W_t^2) \geq 0$ and hence, by \eqref{scal},
\begin{align} 
\scal_{\gamma} 
&= 
\scal_{\gamma_t} + 3 \, \tr(W_t^2) - \tr(W_t)^2 - \tr_{\gamma_t} ( \ddot \gamma_t)  \notag\\
&\geq \scal_{\gamma_t} - \tr_{\gamma_t}(\II_t)^2 - \tr_{\gamma_t} ( \ddot \gamma_t)  . 
\label{scalfin} 
\end{align} 
Using \eqref{eq:II} and \eqref{deform} we compute
\begin{align} 
 \label{secondfund}   \II_t & =   h(\xi) - s  \dot\chi_\delta (t) \cdot (h(\xi)-k(\xi)) + Ct\cdot  g_0(\xi) , \\
  \label{secondder}  \ddot \gamma_t & = 2 s \ddot \chi_\delta(t) \cdot (h(\xi)-k(\xi))  - 2C\cdot g_0(\xi) . 
\end{align} 
Now observe
\begin{equation}
\gamma_t - \gamma_0 = -Ct^2\cdot \gamma_0 -2t\cdot h(\xi) + 2s \chi_\delta(t) \cdot (h(\xi)-k(\xi))  .
\label{eq:gtg0}
\end{equation}
Using $\delta<C^{-2}$, $0\le t\le\sqrt{\delta}$, $0\le s\le 1$, and $|\chi_\delta|\le\tfrac{\delta}{2}$ we get for the coefficient functions
\begin{align}
\big|-Ct^2\big|
\le
C\sqrt{\delta}t
&\le 
t
\le
\sqrt{\delta}, 
\label{eq:coeff1}\\
\big|2s \chi_\delta(t)\big|
&\le
\delta .
\label{eq:coeff2}
\end{align}
In particular, with respect to the $C^2$-norm induced by $\gamma_0$, 
\begin{equation}
\|\gamma_t - \gamma_0\|_{C^2}
\lesssim
t+t+\delta
\le
3\sqrt{\delta} .
\label{eq:gtg02}
\end{equation}
Thus the $\gamma_t$ lie in a uniformly $C^2$-small neighborhood of the compact set of metrics $\{g_0(\xi)\mid \xi\in K\}$.
Hence their scalar curvatures are uniformly (in $t$, $\xi$, $s$, and $C$) bounded.
In particular, 
\begin{equation}
\scal_{\gamma_t} \gtrsim -1 .
\label{eq:scalt_lowerbound}
\end{equation}
Lemma~\ref{comptr} and \eqref{eq:gtg02} imply for sufficiently small $\delta$:
\begin{align}
| \tr_{\gamma_t} ( \II_t) | 
&\le
| \tr_{\gamma_0} ( \II_t) | + | \tr_{\gamma_t} ( \II_t) - \tr_{\gamma_0} ( \II_t) |  \notag\\
&\lesssim
| \tr_{\gamma_0} ( \II_t) | + \|\gamma_t - \gamma_0\|_{\gamma_0} \cdot \|\II_t\|_{\gamma_0}  \notag\\
&\lesssim
| \tr_{\gamma_0} ( \II_t) | + \sqrt{\delta}\cdot \|\II_t\|_{\gamma_0}  .
\label{eq:trtIIt}
\end{align}
For the second summand we find, using \eqref{secondfund} and $|\dot\chi_\delta|\le c_0$,
\begin{align}
\|\II_t\|_{\gamma_0} 
&=
\big\| h(\xi) - s  \dot\chi_\delta (t)\cdot (h(\xi)-k(\xi)) + Ct\cdot  g_0(\xi)\big\|_{\gamma_0} \notag\\
&\le
\big\| h(\xi) \big\|_{\gamma_0}  + \big\| s  \dot\chi_\delta (t)\cdot (h(\xi)-k(\xi))\big\|_{\gamma_0} + \|Ct\cdot  g_0(\xi)\big\|_{\gamma_0} \notag\\
&\lesssim
1 + 1 + C\sqrt{\delta} \notag\\
&\le
3 .
\label{eq:IIt}
\end{align}
For the first term in \eqref{eq:trtIIt} we now find 
\begin{align}
| \tr_{\gamma_0} ( \II_t) | 
&\le 
\sqrt{n-1} \|\II_t\|_{\gamma_0} 
\lesssim 
1.
\label{eq:trIIt}
\end{align}
Inserting \eqref{eq:IIt} and \eqref{eq:trIIt} into \eqref{eq:trtIIt} yields
\begin{equation}
| \tr_{\gamma_t} ( \II_t) | \lesssim 1 .
\label{eq:trIIt_bound}
\end{equation}
In order to control the third term of \eqref{scalfin} we compute, using \eqref{secondder},
\begin{align}
- \tr_{\gamma_t} ( \ddot \gamma_t)
&=
\tr_{\gamma_t}\big(-2s \ddot \chi_\delta(t) \cdot (h(\xi)-k(\xi))  + 2C\cdot g_0(\xi)\big) \notag\\
&=
-2s \ddot \chi_\delta(t) \cdot \tr_{\gamma_t}(h(\xi)-k(\xi)) + 2C\cdot \tr_{\gamma_t}(g_0(\xi)) . 
\label{eq:trgtgtdd}
\end{align}
For the first summand in \eqref{eq:trgtgtdd} we recall that $\tr_{\gamma_0}(h(\xi)-k(\xi)) \geq 0$ and that for $t\in[0,\delta]$ we have $-\frac{2}{\delta} \leq \ddot\chi_\delta(t) \leq 0$.
Hence, by Lemma~\ref{comptr} and \eqref{eq:gtg02},
\begin{align*}
\ddot \chi_\delta(t) \cdot \tr_{\gamma_t}(h(\xi)-k(\xi))
&\le
\ddot \chi_\delta(t) \cdot \tr_{\gamma_t}(h(\xi)-k(\xi)) - \ddot \chi_\delta(t) \cdot \tr_{\gamma_0}(h(\xi)-k(\xi)) \notag\\
&\le
|\ddot \chi_\delta(t)| \cdot |\tr_{\gamma_t}(h(\xi)-k(\xi)) - \tr_{\gamma_0}(h(\xi)-k(\xi))| \notag\\
&\lesssim
|\ddot \chi_\delta(t) | \|\gamma_t-\gamma_0\|_{\gamma_0} \|h(\xi)-k(\xi)\|_{\gamma_0} \notag\\
&\lesssim
\tfrac2\delta (2t+\delta) \notag\\
&\le
6 .
\end{align*}
For $t\in[\delta,\sqrt{\delta}]$ we find
\begin{align*}
\big|\ddot\chi_\delta(t) \cdot \tr_{\gamma_t}(h(\xi)-k(\xi))\big|
&\lesssim
\big|\tr_{\gamma_t}(h(\xi)-k(\xi))\big| \notag\\
&\lesssim
|\tr_{\gamma_0}(h(\xi)-k(\xi))| + \|\gamma_t-\gamma_0\|_{\gamma_0}\|h(\xi)-k(\xi)\|_{\gamma_0} \notag\\
&\lesssim
1 + \sqrt{\delta} .
\end{align*}
Thus we have for all $t\in[0,\sqrt{\delta}]$:
\begin{equation}
\ddot \chi_\delta(t) \cdot \tr_{\gamma_t}(h(\xi)-k(\xi))
\lesssim
1 .
\label{eq:Term1}
\end{equation}
For the second summand in \eqref{eq:trgtgtdd} we obtain
\begin{align*}
|\tr_{\gamma_t}(g_0(\xi))-n+1|
&=
|\tr_{\gamma_t}(g_0(\xi))-\tr_{g_0(\xi)}(g_0(\xi))| \\
&\lesssim
\|\gamma_t - g_0(\xi)\|_{g_0(\xi)} \|g_0(\xi)\|_{g_0(\xi)}\\
&\lesssim
\sqrt{\delta} .
\end{align*}
Thus
\[
|\tr_{\gamma_t}(g_0(\xi))-n+1| \le \tfrac12
\]
for $\delta$ sufficiently small and therefore
\begin{align}
2C\cdot \tr_{\gamma_t}(g_0(\xi))
&\ge
2(n-\tfrac32)C .
\label{eq:Term3}
\end{align}
Inserting \eqref{eq:Term1} and \eqref{eq:Term3} into \eqref{eq:trgtgtdd} yields
\begin{equation}
- \tr_{\gamma_t} ( \ddot \gamma_t)
\gtrsim
C
\label{eq:III_bound}
\end{equation}
for $\delta$ sufficiently small.
Finally, inserting \eqref{eq:scalt_lowerbound}, \eqref{eq:trIIt_bound}, and \eqref{eq:III_bound} into \eqref{scalfin} we find
$$
\scal_{\gamma}
\gtrsim
C
$$
for sufficiently large $C$.
This concludes the proof of the proposition.
\end{proof} 

Now we can prove the main deformation theorem of this section.
For its precise formulation we need two auxiliary functions.
Define $S_1,S_2\colon [0,1]\to[0,1]$ (see Figure~\ref{fig:4}) by 
$$
S_1(t) =
\begin{cases}
1 & \text{ for } 0\le t\le\frac12, \\
2(1-t) & \text{ for } \frac12\le t \le 1,
\end{cases}
$$
and 
$$
S_2(t) =
\begin{cases}
1-2t & \text{ for } 0\le t\le\frac12, \\
0 & \text{ for } \frac12\le t \le 1.
\end{cases}
$$

\begin{figure}[!ht]
\begin{annotate}
{\includegraphics[scale=0.6]{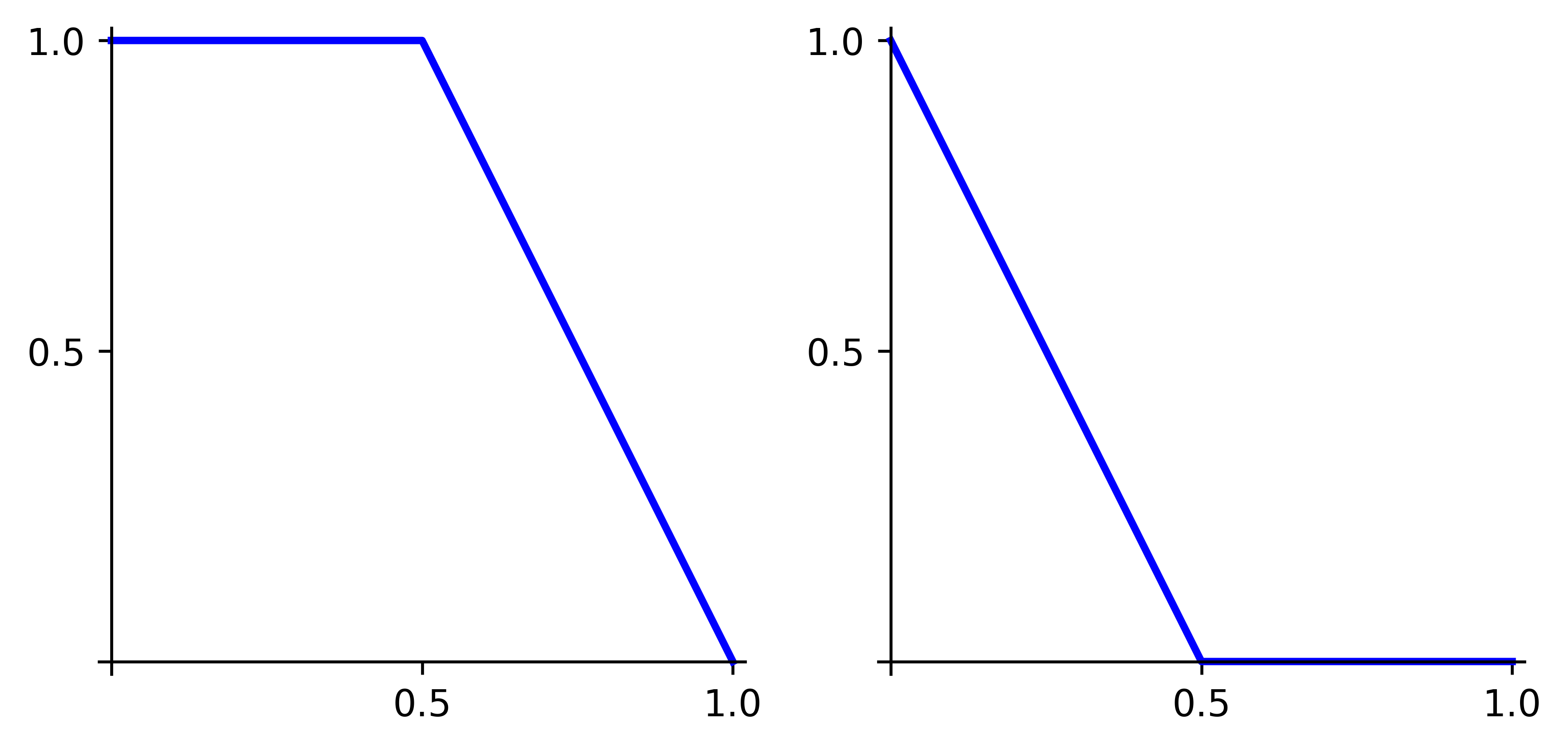}}{0.6}
\note{-2.2,2}{\scalebox{1.4}{$S_1$}}
\note{2.3,2}{\scalebox{1.4}{$S_2$}}
\end{annotate}

\caption{The functions $S_1$ and $S_2$}
\label{fig:4} 
\end{figure}

\begin{theorem}\label{master}
Let $K$ be a compact Hausdorff space and let
$$
g \colon  K \to \RR_{> \sigma} (M)
$$ 
be continuous.
Let $k \colon  K \to  C^{\infty} (\dM ; T^*\dM  \otimes T^*\dM )$ be a continuous family of symmetric $(2,0)$-tensor fields satisfying $\frac{1}{n-1}\tr_{g_0}(k(\xi)) \le H_{g(\xi)}$ for all $\xi\in K$.

Then there exists a constant $C_0>0$ such that for each $C\ge C_0$ and for each neighborhood $\U$ of $\dM$ there is a continuous map 
$$
f \colon  K \times [0,1] \to \RR_{> \sigma} (M)
$$ 
so that the following holds for all $\xi \in K$ and $s \in [0,1]$: 
\begin{enumerate}[(a)] 
   \item  $f(\xi,0) = g(\xi)$; 
   \item  $f(\xi,1)$ is $C$-normal;
   \item  \label{Randmetrikbleibtfest} $f(\xi,s)|_{\dM }=g(\xi)|_{\dM }$, in particular $f(\xi,s)_0 = g(\xi)_0$;
   \item  \label{zweitfund} $\II_{f(\xi,s)}=S_1(s)\II_{g(\xi)} + (1-S_1(s))k(\xi)$, in particular, $\II_{f(\xi,1)}=k(\xi)$;
   \item  \label{Cnormal} if $g(\xi)$ is $\widetilde C$-normal then $f(\xi,s)$ is $C_s$-normal for $C_s=S_2(s)\widetilde C + (1-S_2(s))C$;
   \item \label{secondderivative}  $\ddot f(\xi,s)_0 = S_2(s)\ddot g(\xi)_0 - 2 ( 1-S_2(s)) C g(\xi)_0$; 
   \item  \label{higherderivatives} for $\ell\ge3$ we have $f(\xi,s)^{(\ell)}_0 = S_2(s)\cdot g(\xi)^{(\ell)}_0$;
   \item  $f(\xi,s)=g(\xi)$ on $M\setminus \U$.
\end{enumerate}  
\end{theorem}

\begin{proof}
Let $C_0$ be as in Proposition~\ref{schwer} for the families $g_0(\xi) := g(\xi)_0$, $h(\xi):=\II(\xi)$, and $k(\xi)$.
After possibly increasing $C_0$ we can apply Proposition~\ref{step1} with $C\ge C_0$ to obtain a continuous map $f_1 \colon  K \times [0,1] \to \RR_{> \sigma} (M)$ such that
\begin{enumerate}[(i)]  
   \item $f_1(\xi,0) = g(\xi)$; 
   \item $f_1(\xi,1)$ is $C$-normal; 
   \item \label{cneu} $f_1(\xi,s)|_{\dM } = g(\xi)|_{\dM }$ and $\II_{f_1(\xi,s)} = \II_{g(\xi)}$;  
   \item $\ddot f_1(\xi,s)_0 = (1-s)\ddot g(\xi)_0 -2 s  C g(\xi)_0$; 
   \item \label{higherdiffs2} $f_1(\xi,s)^{(\ell)}_0 = (1-s)g(\xi)^{(\ell)}_0$ for all $\ell\ge3$;
   \item $f_1(\xi,s)=g(\xi)$ on $M\setminus \U$;
\end{enumerate} 
holds  for all $\xi \in K$ and $s \in [0,1]$.
Because of \ref{cneu} we can apply Proposition~\ref{schwer} to the family $f_1(\xi,1)$ with the same constant $C_0$ and obtain another continuous map 
$$
f_2 \colon  K \times [0,1] \to \RR_{> \sigma} (M)
$$ 
so that
\begin{enumerate}[(i)] 
   \setcounter{enumi}{6}
   \item  $f_2(\xi,0) = f_1(\xi,1)$; 
   \item  $f_2(\xi,s)$ is $C$-normal;
   \item  $f_2(\xi,s)|_{\dM } = f_1(\xi,1)|_{\dM } = g(\xi)|_{\dM }$;
   \item  $\II_{f_2(\xi,s)}=(1-s)\II_{f_1(\xi,1)}+sk(\xi) = (1-s)\II_{g(\xi)}+sk(\xi)$;
   \item  $f_2(\xi,s)=f_1(\xi,1)=g(\xi)$ on $M\setminus \U$;
\end{enumerate} 
holds for all $\xi \in K$ and $s \in [0,1]$. 
Letting $f$ be the concatenation of $f_1$ and $f_2$ with respect to the parameter $s$ does the job.
\end{proof}

The following deformation result is much simpler than the previous one.
It may be understood as a genericity statement for strict inequalities for the mean curvature or the second fundamental form.

\begin{proposition}\label{einfach}
Let $K$ be a compact Hausdorff space and let
$$
g \colon  K \to \RR_{>\sigma} (M)
$$ 
be continuous.
Then for each neighborhood $\U$ of $\dM$ there exists a continuous map 
$$
f \colon  K \times [-1,1] \to \RR_{>\sigma}(M)
$$ 
so that the following holds for all $\xi \in K$ and $s \in [-1,1]$: 
\begin{enumerate}[(a)] 
   \item\label{Hposa} $f(\xi,0) = g(\xi)$; 
   \item\label{Randmetrikbleibtfest2} $f(\xi,s)|_{\dM }=g(\xi)|_{\dM }$, in particular $f(\xi,s)_0 = g(\xi)_0$;
   \item\label{Hposb} $\II_{f(\xi,s)}>\II_{g(\xi)}$ if $s>0$ and $\II_{f(\xi,s)}<\II_{g(\xi)}$ if $s<0$;
   \item\label{Hposc} $H_{f(\xi,s)}>H_{g(\xi)}$ if $s>0$ and $H_{f(\xi,s)}<H_{g(\xi)}$ if $s<0$;
   \item\label{Hposd} $f(\xi,s)=g(\xi)$ on $M\setminus\U$.
\end{enumerate} 
\end{proposition}

\begin{proof}
Choose $\eps>0$ such that $U^{g(\xi)}_\eps\subset\U$ and \eqref{eq:metricnearboundary} holds for $g(\xi)$ on $U^{g(\xi)}_\eps$ for all $\xi\in K$.
Let $\psi\colon [0,\infty)\to\R$ be a smooth function with $\psi(t)=t$ for $t$ near $0$ and $\psi(t)=0$ for $t\ge \eps$.
For $\delta>0$ we put
\begin{gather*}
f^{\delta} \colon  K \times [0,1] \to C^{\infty}(M;T^*M \otimes T^*M) , \notag\\
f^{\delta}(\xi,s) = 
\begin{cases}
g(\xi) - s\delta \psi(t) g(\xi)_0 & \mbox{ if }0\le t\le\eps, \\
g(\xi) & \mbox{ if }t\ge\eps.
\end{cases}
\end{gather*}
This defines a family of smooth $(2,0)$-tensor fields on $M$.
For $\delta\to 0$ we have $f^\delta(\xi,s)\to g(\xi)$ in the weak $C^\infty$-topology, uniformly in $\xi$.
In particular, for sufficiently small $\delta$, the $f^\delta(\xi,s)$ are positive definite and hence Riemannian metrics on $M$ and satisfy $\scal_{f^\delta(\xi,s)}>\sigma$.
Properties \ref{Hposa}, \ref{Randmetrikbleibtfest2}, and \ref{Hposd} are obvious.
For the second fundamental form we find
$$
\II_{f^\delta(\xi,s)} = \II_{g(\xi)} +\tfrac12s\delta g(\xi)_0
$$
and hence
\begin{align*}
H_{f^\delta(\xi,s)} 
&= 
\tr_{f^\delta(\xi,s)_0}(\II_{g(\xi)} +\tfrac12s\delta g(\xi)_0)  
= 
\tr_{g(\xi)_0}(\II_{g(\xi)} +\tfrac12s\delta g(\xi)_0) 
=
H_{g(\xi)} + \tfrac{n-1}2 s\delta .
\end{align*}
This proves properties \ref{Hposb} and \ref{Hposc}.
\end{proof}

\begin{remark} \label{small} 
The deformations constructed in Propositions~\ref{step1}, \ref{schwer}, Theorem~\ref{master}, and Proposition~\ref{einfach} are supported in arbitrarily small neighborhoods of the boundary.
They can be chosen ``small'' in another respect as well.

In the deformation in Proposition~\ref{step1} the metric and the second fundamental form remain constant.
Thus the deformation is arbitrarily $C^1$-small in a sufficiently small neighborhood of the boundary.
The conditions of scalar curvature being greater than $\sigma$ and being $\eps$-$C^1$-close to $g(\xi)$ define an open partial differential relation $R(\xi)$ of second order on the space of Riemannian metrics for each $\xi\in K$. 
This family of open partial differential relations depends continuously on $\xi$ in the sense that $\bigcup_{\xi\in K}\{\xi\}\times R(\xi)$ is open in $K\times J^2(T^*M\otimes T^*M)$ where $J^2$ denotes the second jet bundle.
The family version of the flexibility lemma still applies to such $\xi$-dependent partial differential relations as one can see from the proof of Addendum~3.4 in \cite{BH}.
Thus the whole deformation may be chosen arbitrarily $C^1$-small in the sense that given any $\eps>0$ we may arrange that $\|f(\xi,s)-g(\xi)\|_{C^1(M,g(\xi))} < \eps$ for all $\xi$ and $s$, in addition to the properties listed in Proposition~\ref{step1}.

In Proposition~\ref{schwer} the second fundamental form varies but the first fundamental form still remains constant.
From the explicit form of the deformation in \eqref{def:deform} and the bound $|\chi_\delta|\le\frac\delta2$ one sees that the deformation can still be chosen arbitrarily $C^0$-small.

Consequently, in Theorem~\ref{master} we may assume $\|f(\xi,s)-g(\xi)\|_{C^0(M,g(\xi))} < \eps$ for any given $\eps>0$, in addition to the properties listed there.

In the proof of Proposition~\ref{einfach} the deformation is of the form $f(\xi,s)=g(\xi)+s\delta \Psi(\xi)$ where $\Psi(\xi)$ is a smooth compactly supported $(2,0)$-tensor field depending continuously on $\xi$.
Thus, by choosing $\delta$ small, we can make $f(\xi,s)-g(\xi)$ arbitrarily small even in the strong $C^\infty$-topology.
\end{remark}

\begin{remark}
Applying the flexibility lemma with $\xi$-dependent partial differential relations, one may work with a continuous map $\sigma \colon  K \to C^0(M ;  \R)$ instead of a single $\sigma \colon  M \to \R$. 
Given a continous map $g \colon  K \to \RR(M)$ such  that $g(\xi) \in \RR_{> \sigma(\xi)} (M)$ for each $\xi \in K$, Theorem~\ref{master} can be refined so as  to produce a continuous map $f\colon  K \times [0,1] \to \RR(M)$ such that  $f(\xi,s) \in \RR_{> \sigma(\xi)}(M)$ with the properties stated in Theorem~\ref{master} for $\xi \in K$ and $s \in [0,1]$. 
\end{remark} 

\begin{remark} 
 The ``Bending Lemma'' on \cite{Gromov2018}*{p.~705}, which generalizes the well-known ``Doubling trick'' \cite{GL1980}*{Thm.~5.7}, implies that each weak $C^0$-neighborhood of a smooth metric on a manifold with compact boundary contains a smooth metric with the following properties: 
\begin{enumerate}[\myicon] 
\item its scalar curvature is bounded below by the scalar curvature of the given metric minus a prescribed constant, which may be chosen arbitrarily small; 
\item it coincides with the given metric when restricted to the  boundary and outside of a prescribed neighborhood of the boundary, which may be  chosen arbitrarily small;
\item its second fundamental form along the boundary has a prescribed form, provided that its trace is smaller  than the trace of the second fundamental form of the original metric.  \end{enumerate} 

A similar result was obtained in  \cite{BMN}*{Thm.~5}. 
In this respect, our Theorem~\ref{master} gives a deformation theoretic refinement of these approximation results.
\end{remark}

\section{Applications} \label{applications}

We will use Theorem~\ref{master} and Proposition~\ref{einfach} to compare the homotopy types of various spaces of Riemannian metrics.
Throughout this section, let $M$ be a (not necessarily compact) manifold with compact boundary $\dM$.
Let $\sigma\colon M\to\R$ be continuous, let $h_0\colon \dM\to\R$ be a smooth function and let $k$ be a smooth $(2,0)$-tensor field on $\dM$.

Recall that a continuous map $f \colon X \to Y$ between topological spaces is called a {\em weak homotopy equivalence}, if it induces a bijection $\pi_0(X) \cong \pi_0(Y)$   and isomorphisms $\pi_m(X, x) \cong \pi_m(Y, f(x))$ for all $x \in X$ and $m \geq 1$.

\subsection{\texorpdfstring{Subspaces of $\RR_{>\sigma}(M)$ given by geometric conditions on the boundary}{Subspaces of R{>sigma}(M) given by geometric conditions on the boundary}} 
We consider subspaces of $\RR_{>\sigma}(M)$ defined by various conditions on the relative geometry of the boundary.

Theorem~\ref{master} implies, roughly speaking, that inclusions of boundary conditions which are mean-curvature nonincreasing induce weak homotopy equivalences of the corresponding subspaces of metrics with lower scalar curvature bounds. 
Some boundary conditions of particular geometric relevance  to which this reasoning applies are listed in Table~\ref{tab:conditions}.
For any condition ``$*$'' on the boundary as listed in Table~\ref{tab:conditions} we write $\RR_{>\sigma}^*(M)$ for the set of all metrics in $\RR_{>\sigma}(M)$ satisfying $*$.
\begin{table}[!ht]
\begin{tabular}{|c|c|}
\hline
$*$ & condition on $g$ \\
\hline\hline
$H \geq h_0$ & mean curvature satisfies $H_g \geq h_0$ \\
\hline
$H = h_0$ & mean curvature satisfies $H_g = h_0$ \\
\hline
$\II \ge h_0$ & second fundamental form satisfies $\II_g \ge h_0\cdot g_0$ in the sense of bilinear forms \\
\hline
$\II = h_0$ & second fundamental form satisfies $\II_g = h_0\cdot g_0$ \\
\hline
$\II \ge k$ & second fundamental form satisfies $\II_g \ge k$ in the sense of bilinear forms \\
\hline
$\II = k$ & second fundamental form satisfies $\II_g =k$ \\
\hline
\end{tabular}
\medskip

\caption{\label{tab:conditions}Boundary conditions for $g$}
\end{table}

For example, $\RR_{>\sigma}^{H \geq h_0}(M)$ is the space of all smooth metrics $g$ on $M$ with $\scal_g>\sigma$ such that the mean curvature of the boundary satisfies $H_g\ge h_0$.
Furthermore, by $^\nor\RR_{>\sigma}^{*} (M)$ we denote the space of metrics in $\RR_{>\sigma}^{*} (M)$ which are, in addition, $C$-normal for some $C$ (see Definition~\ref{C-normal}). 

\begin{theorem} \label{main} 
Each of the inclusions in 
$$
\xymatrixrowsep{1pc}
\xymatrixcolsep{1pc}
\xymatrix{
     &         &        \RR_{>\sigma}^{H=h_0}(M)  \ar@{^{(}->}[dr]  &  \\
    ^\nor\RR_{>\sigma}^{\II = h_0}(M)  \ar@{^{(}->}[r] &
    \RR_{>\sigma}^{\II = h_0}(M) \ar@{^{(}->}[ur] \ar@{^{(}->}[dr]&  
         &
    \RR_{>\sigma}^{H \geq h_0}(M) \\
    & & \RR_{>\sigma}^{\II \ge h_0}(M) \ar@{^{(}->}[ur] & \\
       ^\nor\RR_{>\sigma}^{\II = k}(M)  \ar@{^{(}->}[r] &
    \RR_{>\sigma}^{\II = k}(M)  \ar@{^{(}->}[r]&  
    \RR_{>\sigma}^{\II \ge k}(M) & 
 }
$$
is a weak homotopy equivalence. 
\end{theorem} 

\begin{proof}
Let $*$ be any of the conditions in $\{ \II=h_0,H=h_0,\II \geq h_0,H \geq h_0 \}$.
We show that the inclusion $^\nor\RR_{>\sigma}^{\II=h_0}(M) \hookrightarrow \RR_{>\sigma}^*(M)$ is a weak homotopy equivalence.

Let ${m } \geq 0$ and $g\colon D^{m }\to \RR_{>\sigma}^*(M)$ be continuous such that $g(\partial D^{m })\subset {} ^\nor\RR_{>\sigma}^{\II=h_0}(M)$.
Here $D^{m }$ is the standard closed ${m }$-ball.
We apply Theorem~\ref{master} with $K=D^{m }$ and $k(\xi)=h_0 g(\xi)$ and obtain a continuous map $f\colon D^{m }\times[0,1]\to\RR_{>\sigma}^*(M)$ such that $f(\xi,s)\in {} ^\nor\RR_{>\sigma}^{\II=h_0}(M)$ for $(\xi,s)\in (\partial D^{m }  \times [0,1]) \cup (D^{m }  \times \{1\})$.

In summary, the pair $(\RR_{>\sigma}^*(M), ^\nor\RR_{>\sigma}^{\II=h_0}(M))$ is ${m }$-connected for all ${m } \geq 0$, compare  \cite{Hatcher}*{p.~346}. 
Hence the inclusion $^\nor\RR_{>\sigma}^{\II=h_0}(M) \hookrightarrow \RR_{>\sigma}^*(M)$ is a weak homotopy equivalence by the long exact sequence for homotopy groups. 

The same argument works for the inclusions $^\nor\RR_{>\sigma}^{\II=k}(M) \hookrightarrow \RR_{>\sigma}^*(M)$ for $*$ in $\{ \II=k,\II \geq k\}$. 
\end{proof} 

Sometimes one needs to compare boundary conditions defined by strict inequalities with their counterparts defined by nonstrict inequalities. 
The proof of the following theorem, which uses the obvious notation, proceeds in exactly the same way as the proof of Theorem~\ref{main}, using Proposition~\ref{einfach} instead of Theorem~\ref{master}.

\begin{theorem} \label{main_suppl} 
Each of the inclusions
\begin{align*} 
\RR_{> \sigma}^{H > h_0}(M) &\hookrightarrow \RR_{> \sigma}^{H \geq h_0}(M), &  \RR_{> \sigma}^{H < h_0}(M) &\hookrightarrow \RR_{> \sigma}^{H \leq h_0}(M),
\\
\RR_{> \sigma}^{\II  > h_0}(M) &\hookrightarrow \RR_{> \sigma}^{\II \geq h_0}(M), & \RR_{> \sigma}^{\II  < h_0}(M) &\hookrightarrow \RR_{> \sigma}^{\II \leq h_0}(M),
\\
\RR_{> \sigma}^{\II > k }(M) &\hookrightarrow \RR_{> \sigma}^{\II \geq k}(M), &  \RR_{> \sigma}^{\II < k}(M) &\hookrightarrow \RR_{> \sigma}^{\II \leq k}(M),
\end{align*} 
is a weak homotopy equivalence. 
\hfill\qed
\end{theorem} 

In the special case $h_0=0$ there is another interesting space one may consider, namely
\[
     \RR_{>\sigma}^{\D}(M)  := \{ g \in  \RR_{>\sigma}(M)\mid g \cup g \text{ is a smooth metric on } \D M\} \, , 
\]
the space of ``doubling'' metrics on $M$ with scalar curvature bounded by $\sigma$ from below. 
Here $\D M = M \cup_{\dM } M$ is the double of $M$.
Writing the metric near the boundary as in \eqref{eq:metricnearboundary}, the doubling condition means $g^{(\ell)}_0=0$ for all odd $\ell$.
This property is also preserved by the deformation in Theorem~\ref{master}.
In particular, we have $\II_g=0$.

\begin{corollary} \label{cor:main} 
Each of the inclusions in 
$$
\xymatrixrowsep{1pc}
\xymatrixcolsep{1pc}
\xymatrix{
    &         &        & \RR_{>\sigma}^{H=0}(M)  \ar@{^{(}->}[dr]  &  \\
    ^\nor\RR_{>\sigma}^{\II = 0}(M)  \ar@{^{(}->}[r] &
    \RR_{>\sigma}^{\D}(M) \ar@{^{(}->}[r] &
    \RR_{>\sigma}^{\II = 0}(M) \ar@{^{(}->}[ur] \ar@{^{(}->}[dr]&  
 &
    \RR_{>\sigma}^{H \geq 0}(M) \\
    & & \RR_{> \sigma}^{\II > 0}(M) \ar@{^{(}->}[r] & \RR_{>\sigma}^{\II \ge 0}(M)  \ar@{^{(}->}[ur]& \RR_{>\sigma}^{H>0}(M) \ar@{^{(}->}[u] \\
  }
$$
is a weak homotopy equivalence. 
\hfill\qed
\end{corollary}

In particular, for any of the inclusions appearing in Corollary~\ref{cor:main}, the subspace is nonempty if the ambient space is nonempty.
For example, if $M$ has a positive scalar curvature metric with $H>0$ then it also has a doubling metric with positive scalar curvature.
This implication is the content of \cite{GL1980}*{Thm.~5.7}.
In fact, we see that assuming $H\ge0$ suffices, compare \cite{Almeida1985}*{Thm.~1.1}. 

\begin{remark} \label{rem:doubling} 
Passing to metric doubles,  the discussion of boundary terms in the APS formula in the proof of Theorem~\ref{thm:nonexist} becomes dispensible. 
For example, let $M$ be an even dimensional compact spin manifold with boundary and let  $M$ be of infinite $K$-area.  
Then the double $\D M = M \cup_{\dM } M$ also has infinite $K$-area because any admissible $E \to M$ satisfying $\| R^E \| < \eps$ extends (by the trivial Hermitian bundle with trivial connection) over the second copy of $M$ to an admissible $E' \to \D M$ with $\| R^{E'} \| < \eps$. 
Furthermore, the double $\D M = M \cup_{\dM } M$ carries a spin structure, equal to the given spin structure on one copy of $M$ and to the opposite spin structure for the opposite orientation on the other copy. 

Now, if $M$ carries a positive scalar curvature metric with $H\ge 0$, then, by Corollary~\ref{cor:main}, it also has a doubling metric of positive scalar curvature.
Hence the manifold $\D M$ has a metric with positive scalar curvature.
This contradicts Theorem~\ref{thm:nonexist} for closed manifolds, which in this case is known, see \cite{G1996}*{Sec.~$5\tfrac{1}{4}$}.
\end{remark}

\subsection{Additional conditions on the boundary metric} \label{additional_conditions} 

We now impose additional intrinsic conditions on the boundary metric.
Let $\X\subset\RR(\dM)$ be any subset.
For example, $\X$ could be defined by any of the following conditions:
\begin{enumerate}[\myicon]
 \item having positive scalar curvature;
 \item being Einstein;
 \item having volume $1$;
 \item coinciding with a particular given metric.
\end{enumerate}
Note that a condition like $H\ge h_0$ cannot be used to define $\X$ because it depends not only on the metric of the boundary but also on the second fundamental form. 
We put 
\begin{align*}
\RR^{\X}(M) 
&:= 
\{g\in\RR(M) \mid g_0\in\X \} , \\
\RR_{>\sigma}^{\X; *}(M) 
&:= 
\RR^{\X}(M) \cap \RR_{>\sigma}^{*}(M) ,
\end{align*}
for any of the conditions $*$ listed in Table~\ref{tab:conditions}.

\begin{theorem} \label{mainwithboundary} 
For any subset $\X\subset\RR(\dM)$ each of the inclusions in 
$$
\xymatrixrowsep{1pc}
\xymatrixcolsep{1pc}
\xymatrix{
    &      &     \RR_{>\sigma}^{\X;H=h_0}(M)  \ar@{^{(}->}[dr]  & \\
    ^\nor\RR_{>\sigma}^{\X;\II = h_0}(M)  \ar@{^{(}->}[r] &
    \RR_{>\sigma}^{\X;\II = h_0}(M) \ar@{^{(}->}[ur] \ar@{^{(}->}[dr]&  
   &
    \RR_{>\sigma}^{\X;H \geq h_0}(M) \\
    & \RR_{> \sigma}^{\X; \II > h_0}(M)  \ar@{^{(}->}[r] & \RR_{>\sigma}^{\X;\II \ge h_0}(M) \ar@{^{(}->}[ur] & \RR_{>\sigma}^{\X;H>h_0}(M) \ar@{^{(}->}[u] \\
    \RR_{>\sigma}^{\X;\II < h_0}(M) \ar@{^{(}->}[r] & \RR_{>\sigma}^{\X;\II \le h_0}(M) &\RR_{>\sigma}^{\X;H < h_0}(M) \ar@{^{(}->}[r] & \RR_{>\sigma}^{\X;H \le h_0}(M) \\
     ^\nor\RR_{>\sigma}^{\X ; \II = k}(M)  \ar@{^{(}->}[r] &
    \RR_{>\sigma}^{\X ; \II = k}(M)  \ar@{^{(}->}[r]&  
    \RR_{>\sigma}^{ \X; \II \ge k}(M) &  \RR_{>\sigma}^{ \X; \II >  k}(M)      \ar@{_{(}->}[l] 
}
$$
is a weak homotopy equivalence. 
\end{theorem} 

\begin{proof}
The proofs of Theorem~\ref{main} and Theorem~\ref{main_suppl} carry over because by property~\ref{Randmetrikbleibtfest} in Theorem~\ref{master} and \ref{Randmetrikbleibtfest2} in Proposition~\ref{einfach} the induced metric on the boundary does not change under the deformation and thus stays in $\X$.
\end{proof}

\begin{remark} 
Given a metric $\gamma_0$ on $\dM $ and setting $\X = \{ \gamma_0\}$ the weak homotopy equivalence  $\RR_{>\sigma}^{\X; H=h_0}(M) \hookrightarrow \RR_{>\sigma}^{\X;H\ge h_0}(M)$ means, roughly speaking, that the mean curvature of $\dM $ can be decreased while preserving the boundary metric and a given lower scalar curvature bound on $M$. 
\end{remark} 

\begin{definition}
Let $h_0\in\R$.
A metric $g$ on $M$ is said to be of $h_0$-\emph{cone type} if it takes the form 
$$
g = dt^2 + (1-th_0)^2\cdot g_0
$$
near the boundary.
Such a metric satisfies $\II_g=h_0\cdot g_0$, i.e.\ it is totally umbilic, and it is $C$-normal with $C=-h_0^2$.
\end{definition}
We write 
\begin{gather*}
\RR_{>\sigma}^{h_0\mathsf{-cone}}(M) := \{g\in\RR_{>\sigma}(M) \mid g \text{ is of $h_0$-cone type}\} \quad\text{ and}\\
\RR_{>\sigma}^{\X;h_0\mathsf{-cone}}(M) := \RR^\X(M) \cap \RR_{>\sigma}^{h_0\mathsf{-cone}}(M).
\end{gather*}

A complete Riemannian manifold $M$ of $h_0$-cone type with $h_0\ge0$ can be extended to a complete Riemannian manifold without boundary by attaching the manifold $(-\infty,0]\times \dM$ with the metric $dt^2 + (1-th_0)^2\cdot g_0$ to $M$ along $\dM$.

In the special case $h_0\equiv 0$ the condition $0$-cone type is commonly known as \emph{product type}, a boundary condition often preferred by topologists.
In this case we may attach a metric cylinder.

\begin{theorem}\label{mainwithboundaryaddendum}
Let $M$ be an $n$-dimensional manifold with compact boundary $\dM$.
Let $\sigma\colon M\to\R$ and $\sigma_0\colon \dM\to\R$ be continuous functions and let $h_0\in\R$.
Assume $\sigma_0\ge \sigma|_{\dM}+(n-1)(n-2)h_0^2$.
Then the inclusion
$$
\RR_{>\sigma}^{\{\scal_{g_0}>\sigma_0\};h_0\mathsf{-cone}}(M) \hookrightarrow {^\nor\RR}_{>\sigma}^{\{\scal_{g_0}>\sigma_0\};\II = h_0}(M)
$$
is a weak homotopy equivalence.
\end{theorem}

\begin{proof}
Let $g\colon D^{\ell }\to {^\nor\RR}_{>\sigma}^{\{\scal_{g_0}>\sigma_0\};\II = h_0}(M)$ be a continuous map with $g(\partial D^{\ell })\subset \RR_{>\sigma}^{\{\scal_{g_0}>\sigma_0\};h_0\mathsf{-cone}}(M)$.
Near $\dM$ we write 
$$
g(\xi) = dt^2 + \big(1-2th_0-C(\xi)t^2\big)\cdot g_0(\xi).
$$
We define, near $\dM$, the deformation
$$
F(\xi,s) := dt^2 + \big(1-2th_0-(1-s)C(\xi)t^2 + sh_0^2t^2\big)\cdot g_0(\xi).
$$
Then $F(\xi,s)_0 = g(\xi)_0$ and $\II_{F(\xi,s)} = h_0\cdot g(\xi)_0 =h_0\cdot F(\xi,s)_0$.
Hence the Weingarten map of the boundary is given by $W_{F(\xi,s)}= h_0\cdot\id$.

For the scalar curvature of $F(\xi,s)$ we find along the boundary by \eqref{scal}:
\begin{align*}
\scal_{F(\xi,s)}|_{\dM}
&=
\scal_{F(\xi,s)_0} + 3 \tr(W_{F(\xi,s)}^2) - \tr(W_{F(\xi,s)})^2 - \tr_{F(\xi,s)_0}(\ddot F(\xi,s)_0) \\
&=
\scal_{g(\xi)_0} + 3 \tr(h_0^2\id) - \tr(h_0\id)^2 - \tr_{g(\xi)_0}((1-s)\ddot g(\xi)_0+2sh_0^2 g_0(\xi)) \\
&=
(1-s)\scal_{g(\xi)}|_{\dM} + s \big[\scal_{g(\xi)_0} + 3(n-1) h_0^2 - ((n-1)h_0)^2-2h_0^2(n-1)\big] \\
&=
(1-s)\scal_{g(\xi)}|_{\dM} + s \big[\scal_{g(\xi)_0} - (n-1)(n-2) h_0^2\big] \\
&>
(1-s)\sigma|_{\dM} + s [\sigma_0- (n-1)(n-2) h_0^2] \\
&\ge
\sigma|_{\dM}.
\end{align*}
By continuity, $\scal_{F(\xi,s)}>\sigma$ on a neighborhood of $\dM$.
Applying the family version of the local flexibility lemma \cite{BH}*{Addendum~3.4} for the partial differential relation given by $\scal>\sigma$ yields a deformation $f$ of $g$ such that $f(\xi,s)\in\RR_{>\sigma}^{\{\scal_{g_0}>\sigma_0\};h_0\mathsf{-cone}}(M)$ for $(\xi,s)\in (\partial D^{\ell }  \times [0,1])\cup D^{\ell }\times\{1\}$.
\end{proof}

Combining Theorems~\ref{mainwithboundary} and \ref{mainwithboundaryaddendum} in the special case $\sigma=\sigma_0=h_0=0$ we get

\begin{corollary} \label{cor:mainwithboundaryaddendum}
Each of the inclusions in 
$$
\xymatrixrowsep{1pc}
\xymatrixcolsep{1pc}
\xymatrix{
    & &   \RR_{>0}^{\{\scal_{g_0}>0\};H=0}(M)  \ar@{^{(}->}[dr] & \\
    ^\nor\RR_{>0}^{\{\scal_{g_0}>0\};\II = 0}(M)  \ar@{^{(}->}[r] &
    \RR_{>0}^{\{\scal_{g_0}>0\};\II = 0}(M) \ar@{^{(}->}[ur] \ar@{^{(}->}[dr]&  
  &
    \RR_{>0}^{\{\scal_{g_0}>0\};H \geq 0}(M) \\
    \RR_{>0}^{\{\scal_{g_0}>0\};0\mathsf{-cone}}(M)  \ar@{^{(}->}[u]& \RR_{>0}^{\{ \scal_{g_0} > 0\} ; \II>0}(M) \ar@{^{(}->}[r] & \RR_{>0}^{\{\scal_{g_0}>0\};\II \ge 0}(M) \ar@{^{(}->}[ur] & \RR_{>0}^{\{\scal_{g_0}>0\};H>0}(M) \ar@{^{(}->}[u] \\
    \RR_{>0}^{\{\scal_{g_0}>0\};\II < 0}(M) \ar@{^{(}->}[r] & \RR_{>0}^{\{\scal_{g_0}>0\};\II \le 0}(M) &\RR_{>0}^{\{\scal_{g_0}>0\};H < 0}(M) \ar@{^{(}->}[r] & \RR_{>0}^{\{\scal_{g_0}>0\};H \le 0}(M)
}
$$
is a weak homotopy equivalence.
\hfill\qed
\end{corollary}

In particular, this tells us that if $M$ carries a positive scalar curvature metric such that the boundary satisfies $H\ge0$ and also has positive scalar curvature then $M$ carries such a metric of product type.

\begin{remark}
If we drop the assumption that the boundary has positive scalar curvature itself, the embedding of positive scalar curvature metrics on $M$ with product structure near the boundary into any of the other spaces of metrics considered here is no longer a weak homotopy equivalence in general.

As an example let $M=D^2\times T^{n-2}$.
We give $D^2$ the metric of a round hemisphere, $T^{n-2}$ a flat metric and $M$ the product metric which we denote by $g$.
Then $g\in\RR_{>0}^\D(M)\subset \RR_{>0}^{\II=0}(M)$, the double of $M$ being $S^2\times T^{n-2}$.

But $M$ cannot carry a positive scalar curvature metric with product structure near the boundary.
If it did, the boundary would inherit a positive scalar curvature metric which is impossible since $\dM=T^{n-1}$.
Thus $\RR_{>0}^{0\mathsf{-cone}}(M)=\emptyset$ and the embedding
$$
\RR_{>0}^{0\mathsf{-cone}}(M) \hookrightarrow \RR_{>0}^{\II=0}(M)
$$
is clearly not a weak homotopy equivalence.
\end{remark}

\subsection{Additional conditions in the interior} 

Given a neighborhood $\dM  \subset \U \subset M$ the deformations of Theorem~\ref{main} and Theorem~\ref{main_suppl} can be assumed to be constant on $M \setminus \U$. 
We can therefore  restrict the spaces in Theorems~\ref{mainwithboundary} and \ref{mainwithboundaryaddendum} and in Corollary~\ref{cor:mainwithboundaryaddendum} to subspaces which are invariant under deformations supported near the boundary.

More precisely, if $\mathscr{Y}$ is an arbitrary subset of $\{ g|_{M\setminus\U} \mid g\in\RR(M)\}$ then after intersecting all spaces in Theorem~\ref{mainwithboundary}, \ref{mainwithboundaryaddendum} or Corollary~\ref{cor:mainwithboundaryaddendum} with $_\mathscr{Y}\RR(M) := \{g\in\RR(M) \mid g|_{M\setminus\U}\in\mathscr{Y}\}$ the resulting embeddings are still weak homotopy equivalences.

Prominent examples of such spaces are
\begin{enumerate}[\myicon]
\item $_\mathscr{Y}\RR(M) = \{g\in\RR(M) \mid g \text{ is complete}\}$;
\item $_\mathscr{Y}\RR(M) = \{g\in\RR(M) \mid g \text{ is incomplete}\}$;
\item $_\mathscr{Y}\RR(M) = \{g\in\RR(M) \mid g \text{ is asymptotically flat}\}$;
\item $_\mathscr{Y}\RR(M) = \{g\in\RR(M) \mid g|_{M\setminus\U} = \bar g|_{M\setminus\U}\}$ where $\bar g\in\RR(M)$ is fixed.
\end{enumerate}

\subsection{Spaces of metrics with mean-convex singularities}

We can apply the deformation of metrics with fixed boundary metrics in order to study metrics with singularities, see \cite{Miao}. 
Let $\hat{M}$ be a smooth manifold and $\Sigma \subset \hat{M}$ be a closed hypersurface with trivial normal bundle so that $\hat{M}  = M_1 \cup_{\Sigma} M_2$, where $M_1$ and $M_2$ are smooth manifolds with compact boundary $\Sigma$. 
Let  $M = M_1 \sqcup M_2$ be the disjoint union of $M_1$ and $M_2$. 
This is a smooth manifold with compact boundary $\Sigma \sqcup \Sigma$. 
Let $\sigma \colon  \hat M \to \R$ be continuous. 
By precomposing with the canonical map $M \to \hat M$ this induces a continuous map $M \to \R$ which we denote by the same letter $\sigma$. 
Consider
\[
      \RR_{> \sigma}^{\Sigma}(\hat M)  := \{ g_1 \sqcup  g_2   \in \RR_{> \sigma} (M) \mid (g_1)_0 = (g_2)_0 , H_{g_1} + H_{g_2} \geq 0\} , 
\]
the space of metrics on $\hat M$ with scalar curvature greater than $\sigma$ and  ``mean convex singularity'' at $\Sigma$. 
This space is equipped with the subspace topology of $\RR_{>\sigma}(M)$. 

Each smooth $g \in \RR_{> \sigma}(\hat M)$ induces $ g_1 \sqcup g_2 \in \RR_{> \sigma}^{\Sigma}(\hat M)$ where $g_j = g|_{M_j}$ for $j=1, 2$. 
Conversely, $g_1 \sqcup g_2 \in \RR_{> \sigma}^{\Sigma}(\hat M)$ induces a smooth metric on $\hat M$ if and only if $(g_1)^{(\ell)}_0 = (-1)^{\ell} \cdot (g_2)^{(\ell)}_0 \in C^{\infty}(\Sigma ; T^* \Sigma \otimes T^* \Sigma)$ for $\ell \geq 0$.

\begin{theorem} \label{desingularize} 
The inclusion $\RR_{> \sigma} (\hat M) \hookrightarrow \RR_{> \sigma}^{\Sigma}(\hat M)$ is a weak homotopy equivalence. 
\end{theorem} 

\begin{proof} 
Let ${m } \geq 0$ and $g = g_1 \sqcup g_2  \colon  D^{m }  \to \RR_{> \sigma}^{\Sigma}(\hat M)$ be continuous such that $g( \partial D^{m } ) \subset \RR_{> \sigma} (\hat M)$.  

We apply Theorem~\ref{master} to $M$ with  $k(\xi) =  - \II_{g_2(\xi)} \sqcup \II_{g_2(\xi)} \in C^{\infty} ( \dM  ; T^* \dM  \otimes T^* \dM )$.
This results in a continuous map $f =  D^{m }  \times [0,1] \to \RR_{> \sigma} ^{\Sigma}(\hat M)$ such that $f(\xi,s) \in  \RR_{> \sigma}(\hat M)$ for $(\xi, s) \in (\partial D^{{m }} \times [0,1]) \cup (D^{m } \times \{1\})$ using  parts~\ref{zweitfund},  \ref{secondderivative} and \ref{higherderivatives} of Theorem~\ref{master}. 
\end{proof} 

In particular, if $\hat M$ has a positive scalar curvature metric which is singular along $\Sigma$ satisfying $H_{g_1}+H_{g_2}\ge0$ as described above then it also has a smooth positive scalar curvature metric.

\subsection{Min-Oo's conjecture} \label{minoo} 

Let $M=S^n_+$ be the $n$-dimensional hemisphere. 
Min-Oo conjectured that if $g$ is a Riemannian metric on $M$ such that
\begin{enumerate}[\myicon]
\item $\scal_g\ge n(n-1)$;
\item $g_0$ coincides with the standard metric $g_{S^{n-1}}$ on $\dM=S^{n-1}$;
\item $\II_g=0$;
\end{enumerate}
then $g$ is isometric to the standard metric $g_{S^n_+}$ of constant sectional curvature, see \cite{M1}*{Thm.~1} or \cite{M2}*{Thm.~4}.

Much to the surprise of many, Brendle, Marques, and Neves showed in \cite{BMN} that this is not true if $n\ge3$.
To construct counterexamples, they proceed in two steps:

\emph{Step~1 (\cite{BMN}*{Thm.~4}):}
Using a refined pertubation analysis they construct a metric $g$ on $M$ such that 
\begin{enumerate}[\myicon]
\item $\scal_g > n(n-1)$;
\item $g=g_{S^n_+}$ along $\dM$;
\item $H_g>0$.
\end{enumerate}

\emph{Step~2 (\cite{BMN}*{Thm.~5 and Cor.~6}):}
Pertubing $g$ further, they find a metric $\hat g$ such that
\begin{enumerate}[\myicon]
\item $\scal_{\hat g} > n(n-1)$;
\item $\hat g=g_{S^n_+}$ along $\dM$;
\item $\II_{\hat g}=0$.
\end{enumerate}

The second step can alternatively be performed using Theorem~\ref{mainwithboundary}.
Putting $\sigma=n(n-1)$, $h_0=0$, and $\X=\{g_{S^{n-1}}\}$ the theorem tells us that the embedding 
$$
\RR_{>n(n-1)}^{\{g_{S^{n-1}}\}; \II=0}(M) \hookrightarrow \RR_{>n(n-1)}^{\{g_{S^{n-1}}\}; H\ge 0}(M)
$$
is a weak homotopy equivalence.
Hence, since $\RR_{>n(n-1)}^{\{g_{S^{n-1}}\}; H\ge 0}(M)$ is nonempty by the first step, then $\RR_{>n(n-1)}^{\{g_{S^{n-1}}\}; \II=0}(M)$ is nonempty as well.
Note that it suffices to know $g_0=g_{S^{n-1}}$ instead of $g=g_{S^n_+}$ along $\dM$.

Even more, if one can find a nontrivial homotopy class in $\RR_{>n(n-1)}^{\{g_{S^{n-1}}\}; H\ge 0}(M)$ then one will obtain a nontrivial homotopy class in the space of counterexamples to Min-Oo's conjecture.

\begin{remark}
Applying the flexibility lemma once more, it is easy to deduce Theorem~5 as stated in \cite{BMN} from our Theorem~\ref{mainwithboundary} although this is not necessary for the second step in the construction of the counterexamples.
\end{remark}

\subsection{Homotopy groups of spaces of metrics with lower scalar curvature and mean curvature bounds} \label{homgr} 

Let $m \geq 0$. 
By   \cite{HSS14}*{Thm.~1.4 and Rem.~1.5} (for  $m =3$ also see~\cite{KKRW}) there is a smooth fiber bundle $F \hookrightarrow P \to S^{m+1}$ with the following properties: 
\begin{enumerate}[\myicon] 
  \item $F$ is a closed connected spin manifold admitting a metric $g_F$ of positive scalar curvature, 
  \item $P$ is a closed spin manifold with $\Ahat[P] := \int_P\Ahat(P) \neq 0$.
\end{enumerate} 

Choose some compact connected spin manifold $W$ with boundary such that $W$ has infinite $K$-area. 
For example, we may take $W$ as the $3$-torus with a small solid torus removed as in Example~\ref{ex:threetorus}. 
Set 
\begin{equation}
M  := W \times F.
\label{eq:DefM}
\end{equation}
This is a compact connected spin manifold with boundary $\dM = \partial W \times F$. 
Let $\sigma \colon  M \to \R$ and $\sigma_0 \colon  \dM  \to \R$ be continuous and let $h_0 \colon \dM\to \R$ be smooth. 

Pick a Riemannian metric $g_W$ on $W$ with $\II_{g_W}=C\cdot (g_W)_0$ for some constant $C > \frac{\dim F + \dim W - 1}{\dim W -1}  \cdot h_0$. 
For any metric $g'_F$ on $F$ the metric $g : = g_W \oplus g'_F$ on $M$ satisfies
 \begin{align} \label{estmean} 
       H_g  = \frac{\tr_g(\II_g)}{\dim W + \dim F - 1}  = \frac{ \tr_{g_W} (\II_{g_W})}{\dim W + \dim F -1} =  \frac{(\dim W - 1) \cdot C  }{\dim W + \dim F -1} > h_0 . 
 \end{align} 
The total space of the bundle $F \to P \to S^{m+1}$ can be written in the form $P = (D^{m+1} \times F) \cup_{\phi} (D^{m+1} \times F)$ for a clutching map $\phi \colon  S^{m} \to \Diff(F)$ such that the adjoint map $S^m \times F \to S^m \times F$, $(\xi, f) \mapsto (\xi, \phi(\xi)(f))$, is smooth.
By compactness of $S^m$, the additivity of scalar curvature in Riemannian products, $\scal_{g_F}>0$, and \eqref{estmean} there exists  $\eps_0 > 0 $ such that for all $0 < \eps \leq \eps_0$ and $\xi \in S^m$ we have $g_W \oplus \eps \cdot \phi(\xi)^*(g_F) \in \RR_{>\sigma}^{\{\scal_{g_0} > \sigma_0\} ; H > h_0}(M)$. 
Consider
\begin{equation}
\omega \colon  S^m \to \RR_{>\sigma}^{\{\scal_{g_0} > \sigma_0\} ; H > h_0}(M), \quad
\omega(\xi) =  g_W \oplus \eps_0  \cdot \phi(\xi)^*(g_F), 
\label{eq:DefOmega}
\end{equation}
and put $\hat g  := g_W \oplus \eps_0 \cdot g_F $. 

\begin{theorem} \label{thm:homotopy} 
Let $M$ be as in \eqref{eq:DefM} and assume $\sigma\ge0$ and $h_0\ge0$.
Then the map $\omega$ constructed in \eqref{eq:DefOmega} represents nonzero classes in   $\pi_{m}\Big( \RR_{>\sigma}^{\{\scal_{g_0} > \sigma_0\} ; H > h_0}(M) , \hat g\Big)$ and in $\pi_{m}\Big( \RR_{> \sigma}^{H > h_0}(M) , \hat g\Big)$.
\end{theorem} 

\begin{proof} 
It is enough to show that $\omega$ represents a nonzero class in $\pi_{m}( \RR_{> 0}^{H > 0}(M) , \hat g)$.
Arguing by contradiction, we assume that there is a continuous homotopy $h \colon  S^m \times [0,1]  \to \RR_{>0}^{H > 0}(M)$ from $\omega$ to the constant map $\xi \mapsto g_N + \eps_0 \cdot g_F$.

Consider the  smooth fiber bundle 
\[
       F \to N \stackrel{\pi}{\to}  S^m \times S^1 , 
 \]
where $N = (S^m \times [0,1] \times F) / \big( (\xi, 0, f)  \sim (\xi, 1,\phi(\xi)(f)   \big)$ is the parametrized mapping cylinder for the family  $(\phi(\xi))_{\xi \in S^m}$, and the bundle projection is given by $[\xi,s,  f ] \mapsto (\xi,[s]) $. 
Then $N$ is a closed connected spin manifold with $\Ahat[N] = \Ahat[P] \neq 0$ because $P$ and $N$ are spin bordant (compare the proof of \cite{HSS14}*{Cor.~1.6}). 

The manifold $W \times N$ is the total space of the  smooth fiber bundle 
\[
      M \to W \times N \stackrel{\pi' }{\longrightarrow } S^m \times S^1
\]
where $\pi' (w,x) = \pi(x)$. 
In the following we denote by $M_b = (\pi')^{-1}(b) \subset W \times N$ the fiber over $b \in S^m \times S^1$. 
Since for $\xi \in S^m$ the map $\id \times \phi(\xi) \colon  ( W \times F ,  \omega(\xi)) \to (W \times F,  \hat g)$ is an isometry, the homotopy $h$ induces a continuous family $(g_b)_{b \in S^m \times S^1}$ of metrics $g_b \in \RR_{> 0}^{H >  0} (M_b)$. 
Since the conditions $\scal>0$ and $H>0$ on the metrics of the fibers $M_b$ are $C^2$-open in the space of smooth Riemannian metrics on the total space $W\times N$, we can, after a $C^2$-small pertubation, assume in addition that $(g_b)_{b \in S^m \times S^1}$ depends smoothly on $b$.

Choose a horizontal distribution on the tangent bundle $T( W \times N)$ with respect to the bundle projection $\pi' \colon  W \times N \to S^m \times S^1$ in such a way that the distribution is tangential to $\partial(W\times N)$ along the boundary.
For  $0 < \lambda \leq 1$ let $g(\lambda)$ be the metric on the total space $W \times N$ determined by this horizontal distribution, the standard metric on $S^m \times S^1$ and the smooth family $(\lambda \cdot g_b)_{b \in S^m \times S^1}$ on the fibers. 
For each $\lambda$ the fibration $\pi'$ is a Riemannian submersion.

By the O'Neill formulas, the scalar curvature of $g(\lambda)$ is given by
$$
\scal_{g(\lambda)} 
=
\scal_{S^m\times S^1}\circ\pi' + \lambda^{-1}\scal_{g_{\pi'(\cdot)}} - \lambda |A|^2 - |T|^2 - |\mathscr{H}|^2 - 2\div_\mathsf{hor}(\mathscr{H}) ,
$$
see \cite{Besse}*{Ch.~9}.
Here $A$ is the curvature of the horizontal distribution, $T$ the second fundamental form of the fibers, $\mathscr{H}$ the unnormalized mean curvature vector field of the fibers, and $\div_\mathsf{nor}$ the horizontal divergence.
Since $\scal_{g_b}>0$, the scalar curvature of $g(\lambda)$ is positive for $\lambda>0$ small enough.

By the choice of horizontal distribution, the boundaries of $W\times N$ and of the fibers $M_b$ have the same interior unit normal vector field $\nu$.
Since $\nu$ is vertical, the second fundamental forms $\II^{\partial(W\times N)}_{g(\lambda)}$ and $\II^{\dM _b}_{g_b}$ satisfy
$$
\II^{\partial(W\times N)}_{g(\lambda)}(U,V)
=
g(\lambda)(\nabla^{g(\lambda)}_UV,\nu)
=
\lambda g_b(\nabla^{\lambda g_b}_UV,\nu)
=
\II^{\dM _b}_{\lambda g_b}(U,V)
=
\lambda^{-1}\II^{\dM _b}_{g_b}(U,V)
$$
for all $U,V$ tangent to $\dM _b$ and 
$$
\II^{\partial(W\times N)}_{g(\lambda)}(X,Y)
=
g(\lambda)(\nabla^{g(\lambda)}_XY,\nu)
=
g_b(A_XY,\nu)
$$
for horizontal $X$ and $Y$.
Since the second fundamental form is symmetric while $A$ is skew-symmetric in $X$ and $Y$, we have $\II^{\partial(W\times N)}_{g(\lambda)}(X,Y)=0$.
Taking traces, we conclude $H_{g(\lambda)} = \lambda^{-1}\frac{\dim M -1}{\dim W + \dim N -1}H_{g_b}>0$ for all $\lambda>0$.

Since $W$ is a compact connected spin manifold which is of infinite $K$-area and $N$ is a closed connected spin manifold with $\Ahat[N] \neq 0$ the existence of a positive scalar curvature metric with strictly mean convex boundary on $W \times N$ is in contradiction to Theorem~\ref{thm:nonexist}.
\end{proof} 

Combined with Theorems~\ref{main} and \ref{mainwithboundary} we deduce:

\begin{corollary} \label{differentboundary} 
For each $m \geq 0$ there exists a compact spin manifold $M$ with nonempty boundary such that for every nonnegative continuous function $\sigma \colon  M \to \R$ and every nonnegative smooth function $h_0 \colon  \partial W \to \R$ the $m$-th homotopy of every space in the diagram 
$$
\xymatrixrowsep{1pc}
\xymatrixcolsep{1pc}
\xymatrix{
       & &  \RR_{> \sigma}^{H=h_0}(M)  \ar@{^{(}->}[dr] & \\
    ^\nor\RR_{> \sigma}^{\II = h_0}(M)  \ar@{^{(}->}[r] &
    \RR_{> \sigma}^{\II = h_0}(M) \ar@{^{(}->}[ur] \ar@{^{(}->}[dr]&  
   &
    \RR_{>\sigma}^{H \geq h_0}(M) \\
    &\RR_{>\sigma}^{\II > h_0}(M) \ar@{^{(}->}[r] & \RR_{>\sigma}^{\II \ge h_0}(M) \ar@{^{(}->}[ur] & \RR_{>\sigma}^{H>h_0}(M) \ar@{^{(}->}[u] 
}
$$
contains nontrivial classes.
The same holds  for all of  these spaces under the additional boundary condition $\X = \{\scal_{g_0} > \sigma_0 \}$ where $\sigma_0 \colon  \dM  \to \R$ is an arbitrary continuous map. 
\hfill\qed
\end{corollary} 

\begin{remark} 
Using the methods in \cite{HSS14}, one can show that for $m \geq 1$ the homotopy classes in Corollary~\ref{differentboundary} are of infinite order and for $m=0$ the spaces have infinitely many path-components. 
\end{remark}


\end{document}